\documentclass[12pt]{amsart}
\usepackage{amsthm,amsfonts,amssymb,amscd,mathrsfs}
\usepackage{bezier,longtable,amstext}
\usepackage{amssymb}

\usepackage[all]{xy}

\newcommand\codim{\text{codim}}

\newcommand{\Pic}{\operatorname{Pic}\nolimits}
\newcommand\pr{\text{pr}}

\newcommand\rk{\text{rk}}

\newcommand\Span{\text{Span}}

\renewcommand{\phi}{\varphi}

\newtheorem{theorem}{Theorem}[section]
\newtheorem{proposition}[theorem]{Proposition}

\newtheorem{lemma}[theorem]{Lemma}
\newtheorem{corollary}[theorem]{Corollary}
\newtheorem{sub}[subsection]{}

\theoremstyle{definition}
\newtheorem{definition}[theorem]{Definition}

\newtheorem{proposition-definition}[theorem]{Proposition-Definition}

\nonstopmode

\begin{document}

\title{On the Barth--Van de Ven--Tyurin--Sato theorem}

\author[I.Penkov]{\;Ivan~Penkov}

\address{
Jacobs University Bremen \\
School of Engineering and Science, Campus Ring 1, 28759 Bremen, Germany}
\email{i.penkov@jacobs-university.de}

\author[Tikhomirov]{\;Alexander S.~Tikhomirov}

\address{
Department of Mathematics\\
State Pedagogical University\\
Respublikanskaya Str. 108
\newline 150 000 Yaroslavl, Russia}
\email{astikhomirov@mail.ru}

\maketitle

\thispagestyle{empty}

\begin{abstract}
The Barth--Van de Ven--Tyurin--Sato Theorem claims that any finite rank vector bundle on
the infinite complex projective space $\mathbf{P}^\infty$ is isomorphic to a direct sum of
line bundles. We establish sufficient conditions on a locally complete linear ind-variety
$\mathbf{X}$ which ensure that the same result holds on $\mathbf{X}$. We then exhibit
natural classes of locally complete linear ind-varieties which satisfy these sufficient conditions.



\end{abstract}

\section{Introduction}\label{sec1}
\label{Introduction}
The Barth--Van de Ven--Tyurin--Sato Theorem claims that any finite rank vector bundle on the infinite complex
projective space $\mathbf{P}^\infty$ is isomorphic to a direct sum  of line bundles. For rank-two bundles this was
established by Barth and Van de Ven in \cite{BV}, and for finite rank bundles it was proved by Tyurin in \cite{T} and
Sato in [S1]. This topic was revived in the more recent papers \cite{DP}, \cite{PT1}, \cite{PT2}, where in particular
the case of twisted ind-varieties was considered.

In the current paper we consider ind-varieties $\mathbf{X}=\underset{\longrightarrow}\lim\ X_m$ given by chains of embeddings of
smooth complete algebraic varieties
$$X_1\stackrel{\phi_1}{\hookrightarrow}X_2\stackrel{\phi_2}{\hookrightarrow}
\dots\stackrel{\phi_{m-1}}{\hookrightarrow}X_m\stackrel{\phi_m}{\hookrightarrow}\dots.$$
We call such ind-varieties \textit{locally complete}. A locally complete
ind-variety $\mathbf{X}=\underset{\longrightarrow}\lim\ X_m$ is \textit{linear} if the map on Picard groups
induced by $\varphi_i$ is a surjection for almost all $i$.
Our main objective is to give a reasonably general sufficient condition for the
Barth--Van de Ven--Tyurin--Sato Theorem to hold on a locally complete ind-variety
$\mathbf{X}$.

In the linear case, besides the results from the 1970-ies and the important results of
Sato [S2],[S3] in which he considers a case when the Barth--Van de Ven--Tyurin--Sato
Theorem no longer holds, some more recent results belong to Donin and Penkov [DP]. In
particular, it is shown in [DP] that the Barth--Van de Ven--Tyurin--Sato Theorem  holds on
any linear direct limit
$\mathbf{G}(\infty)=\underset{\longrightarrow}\lim\ G(k_m,\mathbb{C}^{n_m})$, where
$G(k_m,\mathbb{C}^{n_m})$ denotes the grassmannian of $k_m$-dimensional subspaces in
$\mathbb{C}^{n_m}$, under the assumption that
$\underset{m\to\infty}\lim k_m=\underset{m\to\infty}\lim(n_m-k_m)=\infty$.
It turns out that there is a single isomorphism class of such ind-varieties.
Nevertheless, there are other natural homogeneous ind-varieties on which the Barth--Van de
Ven--Tyurin--Sato Theorem holds but which have not been considered in the literature. This
applies in particular to linear direct limits of isotropic (orthogonal or symplectic)
grassmannians, as well as to direct products of such direct limits.

For this reason we formulate a set of abstract conditions on a linear locally complete
ind-variety $\mathbf{X}$ which ensure that the Barth--Van de Ven--Tyurin--Sato Theorem
(shortly, BVTS Theorem) holds. We then give many examples of ind-varieties
$\mathbf{X}$ satisfying these sufficient conditions. An interesting new class of such ind-varieties consists of direct limits $\mathbf{Y}$ of linear sections $Y_m$ of
$G(k_m,\mathbb{C}^{n_m})$, where
$\underset{\longrightarrow}\lim\ G(k_m,\mathbb{C}^{n_m})=\mathbf{G}(\infty)$.
Another class of ind-varieties on which the Barth--Van de Ven--Tyurin--Sato Theorem holds
are certain ind-varieties of generalized flags, see subsection 6.3.

Probably, there are more general sufficient conditions for the Barth--Van de Ven--Tyurin--
Sato Theorem to hold on locally complete ind-varieties. In addition, for non-linear locally
complete ind-varieties nothing seems to be known beyond the results of [PT2]. Therefore,
providing a sufficient condition for the Barth--Van de Ven--Tyurin--Sato Theorem to hold
on general locally complete ind-varieties remains a project for the future.

{\bf Acknowledgements.} We acknowledge the support and hospitality of the Max Planck
Institute for Mathematics in Bonn where the present paper was conceived. We also acknowledge
partial support from the DFG through Priority Program "Representation Theory" (SPP 1388)
at Jacobs University Bremen. A.S.T. has been financially supported by the Ministry of Education and Science of the Russian Federation.

\vspace{1cm}
\section{Linear ind-varieties. Statement of the main result}\label{linear ind-var}

The ground field is $\mathbb{C}$. We use the term algebraic variety as a synonym for a
reduced Noetherian scheme. If $E$ is a vector bundle (or simply a vector space), $E^*$
stands for the dual bundle (or dual space). We use the standard notation
$\mathcal{O}_{\mathbb{P}^n}(a)$ for the line bundle
$\mathcal{O}_{\mathbb{P}^n}(-1)^{\otimes-a}$, where
$\mathcal{O}_{\mathbb{P}^n}(-1)$ is the tautological bundle on the complex $n$-dimensional
projective space $\mathbb{P}^n$.

Recall that an \textit{ind-variety} is the direct limit
$\mathbf{X}=\underset{\longrightarrow}\lim X_m$ of a chain of morphisms of algebraic varieties
\begin{equation}\label{eq1}
X_1\stackrel{\phi_1}{\to}X_2\stackrel{\phi_2}{\to}\dots\stackrel{\phi_{m-1}}{\to}X_m
\stackrel{\phi_m}{\to}X_{m+1}\stackrel{\phi_{m+1}}{\to}\dots\ .
\end{equation}
Note that the direct limit of the chain (\ref{eq1}) does not change if we replace the sequence $\{X_m\}_{m\ge1}$ by a subsequence $\{X_{i_m}\}_{m\ge1}$, and the morphisms
$\phi_m$ by the compositions
$\tilde{\phi}_{i_m}:={\phi}_{i_{m+1}-1}\circ...\circ{\phi}_{i_m+1}\circ{\phi}_{i_m}$.

Let $\mathbf{X}$ be the direct limit of (\ref{eq1}) and  $\mathbf{X}'$ be the direct limit of a chain
$$
X'_1\stackrel{\phi'_1}{\to}X'_2\stackrel{\phi'_2}{\to}\dots\stackrel{\phi'_{m-1}}{\to}X'_m\stackrel{\phi'_m}{\to}
X'_{m+1}\stackrel{\phi'_{m+1}}{\to}\dots\ .
$$
A {\it morphism of ind-varieties} $\mathbf{f}:\mathbf{X}\to\mathbf{X}'$ is a map from
$\mathbf{X}$ to $\mathbf{X}'$ induced by a collection of morphisms of algebraic varieties
$\{f_m:X_m\to Y_{n_m}\}_{m\ge1}$ such that $\psi_{n_m}\circ f_m=f_{m+1}\circ\phi_m$ for all $m\ge1$. The identity morphism $\mathrm{id}_\mathbf{X}$ is a morphism which coincides with the identity as a set-theoretic map from $\mathbf{X}$ to $\mathbf{X}$.
A morphism $\mathbf{f}:\mathbf{X}\to\mathbf{X}'$ is an \textit{isomorphism} if there exists a morphism
$\mathbf{g}:\mathbf{X}'\to\mathbf{X}$ such that $\mathbf{g}\circ\mathbf{f}=\mathrm{id}_\mathbf{X}$ and
$\mathbf{f}\circ\mathbf{g}=\mathrm{id}_{\mathbf{X}'}$.

In what follows we only consider chains (\ref{eq1}) such that $X_m$ are complete algebraic
varieties, $\underset{m\to\infty}\lim(\dim X_m)=\infty$, and the morphisms $\phi_m$ are
embeddings. We call such ind-varieties \textit{locally complete}. Furthermore, we call a
morphism
$\mathbf{f}:\mathbf{X}=\underset{\to}\lim X_n\to\mathbf{X}'=\underset{\to}\lim X'_n$ of
locally complete ind-varieties an \textit{embedding} if all morphisms $f_m:X_m\to X'_{n_m},\
m\ge1,$ are embeddings.

A {\it vector bundle $\mathbf{E}$ of rank}
$\mathbf{r}\in\mathbb{Z}_{>0}$ on $\mathbf{X}$ is the inverse limit
$\underset{\leftarrow}\lim E_m$ of an inverse system of vector bundles $E_m$ or rank
$\mathbf{r}$ on $X_m$, i.e. a system of vector bundles $E_m$ with fixed
isomorphisms $\psi_m: E_m\cong \phi^*_mE_{m+1}$; here and below $\phi^*$ stands for
inverse image of vector bundles under a morphism $\phi$.
Clearly, $\mathbf{E}|_{X_m}\cong E_m,\ m\ge1$.
In particular, the \textit{structure sheaf} $\mathcal{O}_\mathbf{X}=\underset{\longleftarrow}\lim \mathcal{O}_{X_m}$
of an ind-variety $\mathbf{X}$ is well defined. By the \textit{Picard group} $\text{Pic}\mathbf{X}$ we understand the
group of isomorphism classes of line bundles on $\mathbf{X}$. Clearly, $\text{Pic}\mathbf{X}$ is the inverse limit
$\underset{\longleftarrow}\lim\text{Pic}X_m$ of the inverse system
$\{\phi^*_m:\text{Pic}X_{m+1}\to\text{Pic}X_m\}_{m\ge1}$. In the rest of the paper we automatically assume that all vector bundles considered have finite rank. If $\mathbf{E}$
is a vector bundle on $\mathbf{X}$, $r\mathbf{E}$ stands for the direct sum
$\mathbf{E}\oplus...\oplus\mathbf{E}$ of $r$ copies of $\mathbf{E}$.

A {\it linear ind-variety} is an ind-variety $\mathbf{X}=\underset{\to}\lim X_m$ such that, for each $m\ge1$,
the induced homomorphism of Picard groups $\phi^*_m:\text{Pic}X_{m+1}\to\text{Pic}X_m$ is an epimorphism.
A typical example of a linear ind-variety is the {\it projective ind-space}
$\mathbf{P}^{\infty}$ which is the direct limit of a chain of linear embeddings
$$
\mathbb{P}^{n_1}\stackrel{\phi_1}{\hookrightarrow}\mathbb{P}^{n_2}
\stackrel{\phi_2}{\hookrightarrow}\dots\stackrel
{\phi_{m-1}}{\hookrightarrow}\mathbb{P}^{n_m}\stackrel{\phi_m}{\hookrightarrow}\dots,
$$
for an arbitrary increasing sequence $\{n_m\}_{m\ge1}$ of nonnegative integers. (It is
easy to see that the definition of $\mathbf{P}^{\infty}$ does not depend, up to
isomorphism of ind-varieties, on the
choice of the sequence $\{n_m\}_{m\ge1}$ and the embeddings $\varphi_m$.)
By a {\it projective ind-subspace} of an ind-variety $\mathbf{X}$ we understand the image
of an embedding
$\psi:\mathbf{P}^{\infty}\hookrightarrow\mathbf{X}$.

Another example of a linear ind-variety is the \textit{ind-grassmannian}
$\mathbf{G}(\infty)$ which is the direct limit of a chain of linear embeddings
$$
G(k_1,\mathbb{C}^{n_1})\stackrel{\phi_1}{\hookrightarrow}G(k_2,\mathbb{C}^{n_2})
\stackrel{\phi_2}{\hookrightarrow}...\stackrel{\phi_{m-1}}{\hookrightarrow}G(k_m,
\mathbb{C}^{n_m})\stackrel{\phi_m}{\hookrightarrow}...,
$$
where $G(k_m,\mathbb{C}^{n_m})$ is the grassmannian of $k_m$-dimensional subspaces in an
$n_m$-dimensional vector space and
$\underset{m\to\infty}\lim k_m=\underset{m\to\infty}\lim(n_m-k_m)=\infty$.

Let $\mathbf{X}=\underset{\to}\lim X_m$
be a linear ind-variety such that there is a finite or countable set
$\Theta_{\bf X}$ and a collection
$\{{\bf L}_i=\underset{\leftarrow}\lim L_{im}\}_{i\in\Theta_{\bf X}}$
of line bundles on ${\bf X}$ such that, for any $m$,
$L_{im}\simeq\mathcal{O}_{X_m}$ for all but finitely many indices
$i_1(m),...,i_j(m)$, and the images of
$L_{i_1(m)m},...,L_{i_j(m)m}$ in Pic$X_m$ form a
basis of $\mathrm{Pic}X_m$ which is assumed to be a free abelian group. It is clear that
in this case $\mathrm{Pic}{\bf X}$ is isomorphic to a direct product of infinite cyclic
groups with generators the images of ${\bf L}_i$. We denote by
$\underset{i\in\Theta_{\bf X}}\otimes\mathbf{L}_i^{\otimes a_i}$
the line bundle on {\bf X} whose restriction to $X_m$ equals
$\otimes_iL_{im}^{\otimes a_i}=L_{i_1(m)m}^{\otimes a_1}\otimes...\otimes
L_{i_j(m)m}^{\otimes a_j}$.
We say that $\mathbf{X}$ \textit{satisfies the property} L if, in addition to the above condition,
$H^1(X_m,\otimes_iL^{\otimes a_i}_{im})=0$ for any $m\ge1$ if some $a_i$ is negative.

Let $\mathbf{X}$ satisfy the property L.
For a given $i\in\Theta_{\bf X}$, a smooth rational curve $C\simeq\mathbb{P}^1$ on $\mathbf{X}$ is a
\textit{projective line of the i-th family on} $\mathbf{X}$ (or simply, a {\it line of the i-th family}), if
\begin{equation}\label{Li|Pj}
\mathbf{L}_j|_C\cong\mathcal{O}_{\mathbb{P}^1}(\delta_{ij})\ \ \ for\  \ \ j\in\Theta_{\bf X}.
\end{equation}
By $\mathbf{B}_i$ we denote the set of all projective lines of the i-th family on
$\mathbf{X}$. It has a natural structure of an ind-variety:
$\mathbf{B}_i=\underset{\to}\lim B_{im}$, where
$B_{im}:=\{C\in\mathbf{B}_i\ |\ C\subset X_m\}$ for $m\ge 1.$ For any point $x\in\mathbf{X}$
the subset $\mathbf{B}_i(x)=\{C\in \mathbf{B}_i\ |\ C\ni x\}$
inherits an induced structure of an ind-variety.

Assume that $\mathbf{X}$ satisfies the property L. Then we say that $\mathbf{X}$ \textit{satisfies the property}
A if for any $i\in\Theta_{\bf X}$
there is an ind-variety $\mathbf{\Pi}_i=\underset{\to}\lim \Pi_{im}$
whose points are projective ind-subspaces
$\mathbf{P}^\infty\subset\mathbf{B}_i$,
where
$\Pi_{im}=\{\mathbb{P}^{n_m}\subset B_{im}\ |\ \mathbb{P}^{n_m}=
\mathbf{P}^\infty\cap B_{im}$ for some $\mathbf{P}^\infty\in\mathbf{\Pi}_i\}$,
and, for any point $x\in\mathbf{X}$, the following conditions hold:

(A.i) for each $m\ge1$ such that $x\in X_m$, the sheaf $L_{im}$ defines a morphism
$\psi_{im}:X_m\to\mathbb{P}^{r_{im}}:=\mathbb{P}(H^0(X_m,L_{im})^*)$
which maps the family of lines $B_{im}(x)$ isomorphically to a subfamily of lines in $
\mathbb{P}^{r_{im}}$ passing through the point $\psi_{im}(x)$;

(A.ii) the variety
${\Pi}_{im}(x):=\{\mathbb{P}^{n_m}\in{\Pi}_{im}\ |\ \mathbb{P}^{n_m}\ni x\}$
is connected for any $m\ge 1;$

(A.iii) the projective ind-subspaces $\mathbf{P}^\infty\in\mathbf{\Pi}_i(x):=\underset{\to}\lim \Pi_{im}(x)$ fill
$\mathbf{B}_i(x)$;

(A.iv) for any $d\in\mathbb{Z}_{\ge1}$ there exists a $m_0(d)\in\mathbb{Z}_{\ge1}$ such that, for any $d$-dimensional
variety

$Y$ and any $m\ge m_0(d)$, any morphism $\Pi_{im}(x)\to Y$ is a constant map.\\
In particular, (A.ii) and (A.iii) imply that the varieties $\Pi_{im},$ $B_{im},$ $B_{im}(x)$ are connected.

Let $\mathbf{X}$ satisfy the properties L and A as above. A vector bundle $\mathbf{E}$ on
$\mathbf{X}$ is called $\mathbf{B}_i$-{\it uniform,} if for any projective line
$\mathbb{P}^1\in \mathbf{B}_i$ on $\mathbf{X}$, the restricted bundle
$\mathbf{E}|_{\mathbb{P}^1}$ is isomorphic to
$\oplus_{j=1}^{\mathrm{rk}\mathbf{E}}\mathcal{O}_{\mathbb{P}^1}(k_j)$
for some integers $k_j$ not depending on the choice of $\mathbb{P}^1.$
If in addition all $k_j=0$, then $\mathbf{E}$ is called $\mathbf{B}_i$-linearly trivial.
We call $\mathbf{E}$ {\it uniform} (respectively, {\it linearly trivial}) if it is
$\mathbf{B}_i$-uniform (respectively, $\mathbf{B}_i$-linearly trivial) for any
$i\in\Theta_{\mathbf{X}}$. Moreover,
we say that $\mathbf{X}$ {\it satisfies the property} T if any linearly trivial vector
bundle on $\mathbf{X}$ is trivial.

Our general version of the BVTS Theorem is the following.

\begin{theorem}\label{E on linear X}
Let $\mathbf{E}$ be a vector bundle on a linear ind-variety $\mathbf{X}$.

(i) If $\mathbf{X}$ satisfies the properties $\rm{L}$ and
$\rm{A}$ for for some fixed line bundles
$\{\mathbf{L}_i\}_{i\in\Theta_{\mathbf{X}}}$,
and corresponding families
$\{\mathbf{B}_i\}_{i\in\Theta_{\mathbf{X}}}$ of projective lines on ${\mathbf{X}}$,
then $\mathbf{E}$ has a filtration by vector subbundles
\begin{equation}\label{filtration0}
0=\mathbf{E}_0\subset\mathbf{E}_1\subset...\subset\mathbf{E}_t=\mathbf{E}
\end{equation}
with uniform successive quotients $\mathbf{E}_k/\mathbf{E}_{k-1}$, $k=1,...,t.$

(ii) If, in addition, $\mathbf{X}$ satisfies the property T, then the filtration (\ref{filtration0}) splits and its
quotients are of the form
$$
\mathbf{E}_k/\mathbf{E}_{k-1}\cong\mathrm{rk}(\mathbf{E}_k/\mathbf{E}_{k-1})
(\underset{i\in\Theta_{\mathbf{X}}}\bigotimes\mathbf{L}_i^{\otimes a_{ik}}),
\ \ \ \ a_{ik}\in\mathbb{Z},\ i\in\Theta_{\mathbf{X}},\ 1\le k\le t.
$$
In particular, $\mathbf{E}$ is isomorphic to a direct sum of line bundles.
\end{theorem}

\vspace{1cm}

\section{Proof of the main theorem}\label{sec2}
\vspace{1cm}

\begin{sub}{\bf Preliminaries on vector bundles.}\label{prelim}
\rm

If $C\subset X$ is a smooth irreducible rational curve in an algebraic variety $X$ and
$E$ is a vector bundle on $X$, then by a classical theorem often attributed to Grothendieck,
$\displaystyle E|_C$ is isomorphic to $\bigoplus_i\mathcal{O}_C(\delta_i)$ for some
$\delta_1\geq \delta_2\geq...\geq \delta_{\rk E}$. We call the ordered $\rk E$-tuple
$(\delta_1,...,\delta_{\rk E})$ \emph{the splitting type} of $E|_C$ and denote it
${\rm Split}(E|_C)$. We order splitting types lexicographically, i.e.
$(\delta_1,...,\delta_{\rk E})>(\delta'_1,...,\delta'_{\rk E})$ if
$\delta_1=\delta'_1,...,\delta_{k-1}=\delta'_{k-1},$ $ \delta_k>\delta'_k$ for some
$k,\ 1\le k\le\rk E$.

Let $\mathbf{X}$ be a locally complete linear ind-variety satisfying the properties L
and A, and let $x\in\mathbf{X}$ and $i\in\Theta_{\mathbf{X}}$. In the notation of (A.i), let
$\mathbb{P}^{r_{im}}=\mathbb{P}(H^0(X_m,L_{im})^*)$ and
$y=\psi_{im}(x)=\mathbb{C}u,\ 0\ne u\in H^0(X_m,L_{im})^*$, so that
$B_{im}(x)\subset\mathbb{P}^{r_{im}-1}_y:=\mathbb{P}(H^0(X_m,L_{im})^*/\mathbb{C}u)$.
Fix a projective subspace $\mathbb{P}^{n_m}\subset B_{im}(x)$, where
$\mathbb{P}^{n_m}\in\Pi_{im}(x)$.
Then
$\mathcal{O}_{\mathbb{P}(H^0(X_m,L_{im})^*/\mathbb{C}u)}(1)|_{\mathbb{P}^{n_m}}\simeq\mathcal{O}_{\mathbb{P}^{n_m}}(N(i))$
for some $N(i)>0$. Consider the locally closed subvariety
$Y_m:=\{(z,l)\in\mathbb{P}^{r_{im}}\times\mathbb{P}^{n_m}|z\in l\setminus\{y\}\}$ of
$\mathbb{P}^{r_{im}}\times\mathbb{P}^{n_m}$, and let
$Y_m\hookrightarrow\mathbb{P}^{r_{im}}$
be the embedding induced by the projection
$\mathbb{P}^{r_{im}}\times\mathbb{P}^{n_m}\to\mathbb{P}^{r_{im}}$.
Then $Y_m$ is isomorphic to the total space of the line bundle
$\mathcal{O}_{\mathbb{P}^{n_m}}(N(i))$ (see for instance \cite[Appendix B]{F}) and
\begin{equation}\label{pi*N}
\mathcal{O}_{\mathbb{P}^{r_{im}}}(1)|_{Y_m}\simeq\pi_m^*\mathcal{O}_{\mathbb{P}^{n_m}}(N(i)),
\end{equation}
where $\pi_m:Y_m\to\mathbb{P}^{n_m}$ is the natural projection.
Moreover, by construction we have a commutative diagram of morphisms
\begin{equation}\label{diag Sato}
\xymatrix{
\overline{Y}_m\ar@{_{(}->}[d] & Y_m\ar@{_{(}->}[d]\ar[l]_-{\tau_m}\ar[r]^-{\pi_m} &
\mathbb{P}^{n_m}\ar@{_{(}->}[d] \\
\mathbb{P}^{r_{im}} &
\tilde{\mathbb{P}}^{r_{im}}\ar[l]_-{\phi_y}\ar[r]^-{\pi_y} &
\mathbb{P}^{r_{im}-1}_y,}
\end{equation}
where $\tau_m:Y_m\hookrightarrow\overline{Y}_m:=Y_m\cup\{y\}$ is
the inclusion,
$\varphi_y:\tilde{\mathbb{P}}^{r_{im}}\to\mathbb{P}^{r_{im}}$ is
the blow-up of $\mathbb{P}^{r_{im}}$ with centre at $y$, and
$\pi_y$ is the natural projection which is a
$\mathbb{P}^1$-bundle. In addition, we have an open embedding
$$
\iota_m:Y_m\hookrightarrow\tilde{Y}:=\tilde{\mathbb{P}}^{r_{im}}\times_{\mathbb{P}^{r_{im}-1}_y}\mathbb{P}^{n_m}
$$
and projections
$\overline{Y}_m\overset{\tilde{\tau}_m}\leftarrow\widetilde{Y}_m\overset{\tilde{\tau}_m}\to\mathbb{P}^{n_m}$
such that
\begin{equation}\label{iota}
 \tau_m=\tilde{\tau}_m\circ\iota_m,\ \ \ \pi_m=\tilde{\pi}_m\circ\iota_m.
\end{equation}
By (A.i) $\psi_{im}:\psi_{im}^{-1}(\overline{Y}_m)\to\overline{Y}_m$ is an isomorphism. Hence we may
consider $\mathbf{E}|_{\psi_{im}^{-1}(\overline{Y}_m)}$ as a vector bundle on $\overline{Y}_m$ and denote it by
$\mathbf{E}|_{\overline{Y}_m}$. We also set $\mathbf{E}|_{Y_m}:=\tau_m^*(\mathbf{E}|_{\overline{Y}_m})$.

For an arbitrary projective line $\mathbb{P}^1\subset\mathbb{P}^{n_m}$, we consider the surface
$S=S(x,\mathbb{P}^1):=\pi_y^{-1}(\mathbb{P}^1)$ with natural
projections $ \pi_S:=\pi_y|_S:S\to\mathbb{P}^1$ and
$\sigma_S:=\phi_y|_S:S\to\mathbf{X}$. It follows from (\ref{pi*N})
that $S$ is a surface of type $F_{N(i)}$.

Let $\mathbf{E}$ be a vector bundle of rank $\mathbf{r}$ on $\mathbf{X}.$ For any $i\in \Theta_{\mathbf{X}}$ and
$x\in\mathbf{X}$ we set $C(i):=c_1(\mathbf{E}|_l)\in\mathbb{Z}$, where $c_1$ stands for first Chern class and
$l\in\mathbf{B}_i(x).$ Since $\mathbf{B}_i(x)$ is connected, $C(i)$ is well defined. Furthermore, we have
 $\delta_1(\mathbf{E}|_l)\ge C(i)/\mathbf{r}\ge \delta_B(\mathbf{E}|_l)$. Hence there are well-defined integers
$$
\delta_1^{\rm min}:=\underset{l\in\mathbf{B}_i(x)}\min~ \delta_1(\mathbf{E}|_l),\ \ \ \
\delta_{{\rm rk}E}^{\rm max}:=\underset{l\in\mathbf{B}_i(x)}\max~ \delta_{{\rm rk}E}(\mathbf{E}|_l),
$$
and there exist lines $l_{\min},l_{\max}\in\mathbf{B}_i(x)$ such that
$\delta_1(\mathbf{E}|_{l_{\min}})=\delta_1^{\min},$
$\delta_{{\rm rk}E}(\mathbf{E}|_{l_{\max}})=\delta_{{\rm rk}E}^{\max}.$
The inequality $C(i)\ge\delta_1^{\min}+(\mathbf{r}-1)\delta_{\mathrm{rk}E}(\mathbf{E}|_{l_{\min}})$ implies
\begin{equation}\label{ineq'}
\delta_1^{\min}-\delta_{{\rm rk}E}(\mathbf{E}|_{l_{\min}})\le\delta_1^{\min}-C(i)/(\mathbf{r}-1),\ \ \
\delta_1(\mathbf{E}|_{l_{\max}})-\delta_{\mathrm{rk}E}^{\max}\ge C(i)/(\mathbf{r}-1)-\delta_{\mathrm{rk}E}^{\max}.
\end{equation}

Fix $\mathbf{P}^{\infty}\in\mathbf{B}_i(x)$ and $l_{\min}\in\mathbf{P}^{\infty}$.
For an arbitrary point $l_0\in\mathbf{P}^{\infty}\smallsetminus\{l_{\min}\}$
consider the line
$\mathbb{P}^1=\Span(l_0,l_{\min})$ in $\mathbf{P}^{\infty}$  and the corresponding surface
$S=S(x,\mathbb{P}^1)$ together with the vector bundle $E_S:=\sigma_S^*\mathbf{E}$ on $S$.
For a general point $l\in\mathbb{P}^1$ ($l$ is a line on $ \mathbf{X}$), the first inequality in (\ref{ineq'}) implies
\begin{equation}\label{ineq1'}
\delta_{\mathrm{gen}}:=\delta_1(\mathbf{E}|_l)-\delta_{{\rm rk}E}(\mathbf{E}|_l) \le\delta_1^{{\rm min}}-C(i)/(\mathbf{r}-1).
\end{equation}

The following lemma is a straightforward consequence of a result of Tyurin.

\begin{lemma}\label{lemma31}
There exist polynomials $P_A$, $P_B\in\mathbb{Q}[x_1,...,x_6]$ such that for any
$l_0\in\mathbf{P}^{\infty}\smallsetminus\{l_{\min}\}$
\begin{equation}\label{ineq A}
\delta_1(\mathbf{E}|_{l_0})\le P_A(\mathbf{r},\delta_1^{\min},C(i),N(i),c_1^2(E_S),c_2(E_S))=:P_A(\mathbf{E},i),
\end{equation}
\begin{equation}\label{ineq B}
\delta_{{\rm rk}E}(\mathbf{E}|_{l_0})\ge P_B(\mathbf{r},\delta_{\mathrm{rk}E}^{\max},C(i),N(i),c_1^2(E_S),c_2(E_S))=:P_B(\mathbf{E},i),
\end{equation}
where $c_1^2(E_S)$ and $c_2(E_S)$ are considered as integers.
\end{lemma}

\begin{proof}
By construction, $S$ is a surface of type $F_{N(i)}.$ Hence, repeating for the vector bundle $E_S$ the
proof of Lemma 5 from [T, Ch.~2, \S1] we obtain that there exists a
polynomial $f\in\mathbb{Q}[x_1,...,x_6]$ such that
$\delta_1(\mathbf{E}|_{l_0})\le f(\mathbf{r},\delta_1^{\min},\delta_{\mathrm{gen}},N(i),c^2_1(E_S),c_2(E_S))$.
Thus, in view of (\ref{ineq1'}), there exists a polynomial $P_A\in\mathbb{Q}[x_1,...,x_6]$ satisfying (\ref{ineq A}).
The proof of (\ref{ineq B}) is similar.
\end{proof}

The next proposition employs in a crucial way results of E. Sato. Fix $i\in \Theta_{\mathbf{X}}$, $x\in\mathbf{X}$
and $\mathbb{P}^{n_m}\in\Pi_{im}(x)$ for a large enough $m$. In view of (\ref{ineq A}) there exists a maximal (with respect to lexicographic order) splitting type
$S_i(\mathbf{E},\mathbb{P}^{n_m}):=\max\limits_{l\in \mathbb{P}^{n_m}}\mathrm{Split}(\mathbf{E}|_l)$.

\begin{proposition}\label{proposition3.2}
The maximal splitting type $S_i(\mathbf{E},\mathbb{P}^{n_m})$
depends only on the pair $(\mathbf{E},i)$, i.e. $S_i(\mathbf{E},\mathbb{P}^{n_m})$
does not depend on $x$ and on $\mathbb{P}^{n_m}\in\Pi_{im}(x)$.
\end{proposition}

\begin{proof}
Set
$$
M_i(\mathbb{P}^{n_m}):=\{l\in\mathbb{P}^{n_m}\ |\ \mathrm{Split}(\mathbf{E}|_l)=S_i(\mathbf{E},\mathbb{P}^{n_m})\}.
$$
The semicontinuity of $\mathrm{Split}(\mathbf{E}|_l)$ implies that $M_i(\mathbb{P}^{n_m})$ is a closed subvariety of
$\mathbb{P}^{n_m}$. Moreover, Lemma \ref{lemma31}
together with \cite[Ch. 2, \S2, Lemmas 3 and 4]{T} yields the inequality
\begin{equation}\label{codim A}
\codim_{\mathbb{P}^{n_m}}M_i(\mathbb{P}^{n_m})\le\mathbf{r}(\mathbf{r}-1)(P_A(\mathbf{E},i)-
P_B(\mathbf{E},i)).
\end{equation}

Consider the upper row of the diagram (\ref{diag Sato}). Since the right-hand side of
(\ref{codim A}) is constant with respect to $m$, for large enough $m$ we have
\begin{equation}\label{codim<}
\codim_{\mathbb{P}^{n_m}}M_i(\mathbb{P}^{n_m})<\min(n_m-\mathbf{r},(n_m-2\mathbf{r}^2)/2).
\end{equation}
Also, clearly for large enough $m$
\begin{equation}\label{1codim<}
\codim_{\overline{Y}_m}\{x\}=\codim_{\overline{Y}_m}(\overline{Y}_m\smallsetminus Y_m)>\mathbf{r}.
\end{equation}

The inequality (\ref{1codim<}) shows that
\begin{equation}\label{0ck=}
c_k(\mathbf{E}|_{Y_m})=\tau_m^*c_k(\mathbf{E}|_{\overline{Y}_m}),\ \ \ 0\le k\le\mathbf{r},
\end{equation}
where $c_k(\cdot)$ stands for $k$-th Chern class.
Moreover, since the Chow group $A^k(\cdot)$ of codimension $k$ pulls back isomorphically to the total space of any
vector bundle, we have
\begin{equation}\label{ck=}
c_k(\mathbf{E}|_{Y_m})=\pi_m^*(c_kH^k),\ \ 0\le k\le\mathbf{r},
\end{equation}
where $H$ is the class of a hyperplane divisor on $\mathbb{P}^{n_m}$ and
$c_1,...,c_{\mathbf{r}}$ are integers.
It is essential to note that the obvious compatibility of the morphisms $\pi_m$ for varying $m$ and the functoriality of
Chern classes imply that these integers do not depend
on $x$ and on $\mathbb{P}^{n_m}\in\Pi_{im}(x)$. Note also that (\ref{pi*N}), (\ref{iota}) and
the equalities (\ref{0ck=}) and (\ref{ck=}) imply
\begin{equation}\label{again ck=}
\iota_m^*c_k(\tilde{\tau}_m^*(\mathbf{E}\otimes\mathbf{L}_i^{-a}|_{\overline{Y}_m}))=
\tau_m^*c_k(\mathbf{E}\otimes\mathbf{L}_i^{-a}|_{\overline{Y}_m})=
c_k(\mathbf{E}|_{Y_m}\otimes\pi_m^*\mathcal{O}_{\mathbb{P}^{n_m}}(-N(i)aH)),\ \ \
0\le k\le\mathbf{r},\ a\in\mathbb{Z}.
\end{equation}

Next, consider the polynomial
\begin{equation}\label{poly h}
 h(t)=\sum_{k=0}^{\mathbf{r}}c_k(-t)^{\mathbf{r}-k}\in\mathbb{Z}[t]
\end{equation}
where the coefficients $c_k$ are the integers introduced above. Following closely an idea of Sato, we will now argue that the roots of $h(t)$ constitute a constant multiple of the maximal splitting type 
$S_i(\mathbf{E},\mathbb{P}^{n_m})$.
More precisely, let $a_1>...>a_\alpha$, $\alpha\le\mathbf{r}$, be the distinct elements of
$S_i(\mathbf{E},\mathbb{P}^{n_m})$ of respective multiplicities
$r_1,...,r_\alpha$ in $S_i(\mathbf{E},\mathbb{P}^{n_m})$.
Then we claim that the roots of $h(t)$ are $N(i)a_1,...,N(i)a_{\alpha}$ of respective multiplicities $r_1,...,r_{\alpha}$.

The argument in \cite[pp. 138-139]{S1} shows that in order to prove this claim of $h(t)$ it
suffices to establish the vanishing of
$c_k(\mathbf{E}|_{Y_m}\otimes\pi_m^*\mathcal{O}_{\mathbb{P}^{n_m}}(-N(i)a_jH))$ for
$\mathbf{r}-r_j+1\le k\le\mathbf{r},\ 1\le j\le\alpha$.
By (\ref{again ck=}) it is enough to prove the vanishing of
$c_k(\tilde{\tau}_m^*(\mathbf{E}\otimes\mathbf{L}_i^{-a_j}|_{\overline{Y}_m}))$.
However, the proof of this fact is practically the same as in \cite{S1}.
Namely, one defines inductively vector bundles
$F_1:=\tilde{\phi}_m^*\mathbf{E}|_{\tilde{\pi}_m^{-1}(M_i(\mathbb{P}^{n_m}))},\
F_2,...,F_{\alpha}$ such that $\rk F_j=\sum_{p=j}^{\alpha}r_p$
on $\tilde{\pi}_m^{-1}(M_i(\mathbb{P}^{n_m}))$
which fit into the exact triples
\begin{equation}\label{exact F}
0\to r_j\mathcal{O}_{\tilde{\pi}_m^{-1}(M_i(\mathbb{P}^{n_m}))}\to
F_j\otimes(\mathbf{L}_i^{a_{j-1}-a_j}|_{\tilde{\pi}_m^{-1}(M_i(\mathbb{P}^{n_m}))})\to F_{j+1}\to0, \ \ \
1\le j\le\alpha-1,
\end{equation}
where $a_0:=0$. Using (\ref{codim<}) and applying the argument from
\cite[p. 139]{S1} to the triples (\ref{exact F}), we obtain
$c_k(\tilde{\tau}_m^*(\mathbf{E}\otimes\mathbf{L}_i^{-a_j}|_{\overline{Y}_m}))=0$ as desired.

Since $h(t)$ is independent of $x$ and $\mathbb{P}^{n_m}\in\Pi_{im}(x)$, the same applies to 
$S_i(\mathbf{E},\mathbb{P}^{n_m})$, i.e. the proposition is proved.
\end{proof}

\end{sub}

\vspace{5mm}
\begin{sub}{\bf Proof of Theorem 2.1.}\label{proof Thm 2.2}
\rm

\vspace{2mm}

\begin{proof}
According to (\ref{codim<}), the dimension of $M_i(\mathbb{P}^{n_m})$ is
greater than half of the dimension of $\mathbb{P}^{n_m}$, hence the varieties
$M_i(\mathbb{P}^{n_m})$ are connected for large enough $m$.
Consider the variety $\Gamma_{im}(x):=\{(l,\mathbb{P}^{n_m})\in B_{im}(x)\times \Pi_{im}(x)\ |\ l\in
M_i(\mathbb{P}^{n_m})\}$ with projections
${B}_{im}(x)\overset{p_1}\leftarrow {\Gamma}_{im}(x)\overset{p_2}\to {\Pi}_{im}(x)$.
Since $M_i(\mathbb{P}^{n_m})=p_2^{-1}(\mathbb{P}^{n_m})$
is connected for any $\mathbb{P}^{n_m}\in {\Pi}_{im}(x)$
and ${\Pi}_{im}(x)$ is connected by (A.ii), it follows that $\Gamma_{im}(x)$ is connected.
By definition, $\Gamma_{im}(x)$ is described as
\begin{equation}\label{Aim(x)}
\Gamma_{im}(x)=\underset{\mathbb{P}^{n_m}\in\Pi_{im}(x)}\sqcup M_i(\mathbb{P}^{n_m}),\ \ \ \ m\ge1.
\end{equation}

Similarly, consider the varieties
$\Gamma_{im}:=\{(x,l,\mathbb{P}^{n_m})\in X_m\times B_{im}\times\Pi_{im}\ |\ \ (l,\mathbb{P}^{n_m})\in\Gamma_{im}(x)\}$.
By construction,
$$
\Gamma_{im}=\underset{x\in X_m}\sqcup\Gamma_{im}(x),
$$
so that each $\Gamma_{im}$ is connected.
Moreover, there is a well-defined ind-variety  $\mathbf{\Gamma}_i:=\underset{\to}\lim\Gamma_{im}$.

Let $(x,l,\mathbb{P}^{n_m})\in\mathbf{\Gamma}_i$. If $\delta_1^{{\rm max}}$ is the maximal entry of
$\mathrm{Split}(\mathbf{E}|_l)=S_i(\mathbf{E},\mathbb{P}^{n_m})$, there is a well-defined subbundle
$\mathbf{E}_1(l)$ of $\mathbf{E}|_l$:
\begin{equation}\label{E1(E,l)}
\mathbf{E}_1(l):={\rm im}(H^0(l,\mathbf{E}|_l(-\delta_1^{{\rm max}}))\otimes
\mathcal{O}_l\overset{ev}\to\mathbf{E}|_l(-\delta_1^{{\rm max}}))\otimes
\mathcal{O}_l(\delta_1^{{\rm max}}).
\end{equation}
Set $\mathbf{r}_1:={\rm rk}\mathbf{E}_1(l)$ and consider the relative grassmannian
$\rho_1:\ \mathbf{G}(\mathbf{r}_1,\mathbf{E})\to\mathbf{X}$. According to Proposition \ref{proposition3.2},
$\delta_1^{{\rm max}}$ and $\mathbf{r}_1$ do not depend on the point
$(x,l,\mathbb{P}^{n_m})\in\mathbf{\Gamma}_i$. Thus there is a morphism of ind-varieties
\begin{equation}\label{morphism f1}
\mathbf{f}_{i1}:\mathbf{\Gamma}_i\to\mathbf{G}(\mathbf{r}_1,\mathbf{E}),\ \ \ (x,l)\mapsto\mathbf{E}_1(l)|_x,
\ \ \ x\in\mathbf{X}.
\end{equation}
Since $\mathbf{f}_{i1}(\Gamma_{im}(x))\subset\rho_1^{-1}(x)=G(\mathbf{r}_1,\mathbf{E}|_x)$,
by (\ref{Aim(x)}) we have
$$
\mathbf{f}_{i1}(M_i(\mathbb{P}^{n_m}))\subset G(\mathbf{r}_1,\mathbf{E}|_x),\ \ \
\mathbb{P}^{n_m}\in\Pi_{im}(x).
$$

According to (\ref{codim A})
$\codim_{\mathbb{P}^{n_m}}M_i(\mathbb{P}^{n_m})$
is bounded as $m\to\infty$. This means that, for large enough $m$, the morphism
$\mathbf{f}_{i1}:M_i(\mathbb{P}^{n_m})\to G(\mathbf{r}_1,\mathbf{E}|_x)$
satisfies the conditions of \cite[Prop. 3.2]{S1}, in which we set
$n=n_m,\ X=M_i(\mathbb{P}^{n_m}),\ Y=G(\mathbf{r}_1,\mathbf{E}|_x)$
and $f=\mathbf{f}_{i1}$. By this proposition,
$\mathbf{f}_{i1}|_{M_i(\mathbb{P}^{n_m})}$
is a constant map, hence it induces a morphism
$$
\phi_{i1}(x):\ \Pi_{im}(x)\to G(\mathbf{r}_1,\mathbf{E}|_x),\ \ \
\mathbb{P}^{n_m}\mapsto\mathbf{f}_{i1}(M_i(\mathbb{P}^{n_m})),\ \ \ \ x\in\mathbf{X}.
$$
Now (A.iv) implies that there exists a positive integer $m_1$ such that the morphism $\phi_{i1}(x)$
is a constant map for any $m\ge m_1.$
We thus obtain that (\ref{morphism f1}) induces a constant morphism
$$
\boldsymbol{\phi}_{i1}(x):\mathbf{\Pi}_i(x)\to G(\mathbf{r}_1,\mathbf{E}|_x).
$$

Consider the ind-variety $\mathbf{\Sigma}_i=\underset{\to}\lim \Sigma_{im}$,
where
$\Sigma_{im}:=\{(x,\mathbb{P}^{n_m})\in X_m\times\Pi_{im}\ |\ \mathbb{P}^{n_m}\in\Pi_{im}(x)\}$,
and let
$\mathbf{p}_i:\mathbf{\Sigma}_i\to\mathbf{X}$ be the natural projection with fibre
$\mathbf{\Pi}_i(x)$, $x\in\mathbf{X}$. The above constant morphisms
$\boldsymbol{\phi}_{i1}(x)$
extend to a morphism
$\boldsymbol{\phi}_{i1}:\mathbf{\Sigma}_i\to\mathbf{G}(\mathbf{r}_1,\mathbf{E})$
which is constant on the fibres of $\mathbf{p}_i$.
In addition, the morphism $\mathbf{f}_{i1}|_{\mathbf{M}_i(x)}$ is a constant map.
We thus obtain a well-defined morphism
\begin{equation}\label{morphism Fi10}
\boldsymbol{\Phi}_{i1}:\ \mathbf{X}\to\mathbf{G}(\mathbf{r}_1,\mathbf{E}),\ \
x\mapsto\mathbf{f}_{i1}(\mathbf{\Gamma}_i(x)).
\end{equation}

Let $\boldsymbol{\mathcal{S}}$ be the tautological bundle of rank $\mathbf{r}_1$ on
$\mathbf{G}(\mathbf{r}_1,\mathbf{E})$. Set
$\mathbf{E}_{1i}:=\boldsymbol{\Phi}^*_{i1}\boldsymbol{\mathcal{S}}$.
It follows now from (\ref{E1(E,l)}), (\ref{morphism f1}) and (\ref{morphism Fi10}) that $\mathbf{E}_{1i}$ is a
subbundle of $\mathbf{E}$ such that
$$
\mathbf{E}_{1i}|_l=\mathbf{E}_1(l)\simeq\mathbf{r}_1\mathcal{O}_{l}(\delta_1^{{\rm max}}),
\ \ \ \ l\in\mathbf{M}_i.
$$
Using the semicontinuity of $\dim H^0(l,\mathbf{E}_{1i}(-\delta_1^{{\rm max}})|_l),$ one checks immediately that the
last inequality is true for any $l\in\mathbf{B}_i.$

Applying the above argument to the quotient $\mathbf{E}'=\mathbf{E}/\mathbf{E}_{1i}$ etc., we obtain a filtration of
the bundle $\mathbf{E}$
$$
0\subset \mathbf{E}_{1i}\subset \mathbf{E}_{2i}\subset ... \subset \mathbf{E}_{\alpha i}=\subset \mathbf{E}_{1i}
$$
with $\mathbf{B}_i$-uniform
successive quotients
$\mathbf{F}_{ki}=\mathbf{E}_{ki}/\mathbf{E}_{k-1,i}.$

Fix now $j\in\Theta_{\mathbf{X}},$ $i\ne j$.
By applying the same procedure to all bundles $\mathbf{F}_{ki},$ we obtain a bundle filtration of $\mathbf{E}$ whose
quotients are $\mathbf{B}_i$-uniform and $\mathbf{B}_j$-uniform. After finitely many iterations we finally obtain a
filtration
\begin{equation}\label{*}
0=\mathbf{E}_0\subset\mathbf{E}_1\subset...\subset\mathbf{E}_s=\mathbf{E}
\end{equation}
of $\mathbf{E}$ with uniform successive quotients. This yields (i).

Note that any uniform vector bundle on $\mathbf{X}$ becomes linearly trivial after
twisting by an appropriate line bundle. This means that each successive quotient
$\mathbf{E}_k/\mathbf{E}_{k-1}$ is isomorphic to
$\mathbf{M}_k\otimes\mathbf{F}_k$ where
$\mathbf{M}_k$ is a line bundle and $\mathbf{F}_k$ is linearly trivial.
In addition, assume that the property T is satisfied. Then the bundles $\mathbf{F}_k$ are trivial, i.e.
$$
\mathbf{E}_k/\mathbf{E}_{k-1}\simeq
\mathrm{rk}(\mathbf{E}_k/\mathbf{E}_{k-1})\mathbf{M}_k,\ \ \ 1\le k\le s.
$$
Furthermore, for $p<k$
$$
\mathrm{Ext}^1(\mathbf{M}_k,\mathbf{M}_p)=
H^1(\mathbf{X},\mathbf{M}_k^*\otimes\mathbf{M}_p),
$$
and, according to a well known fact \cite[Theorem 4.5]{H3} (see also \cite[Proposition 10.3]{DPW}),
$$
H^1(\mathbf{X},\mathbf{M}_k^*\otimes\mathbf{M}_p)=
\underset{\leftarrow}\lim H^1(X_m,(\mathbf{M}_k^*\otimes\mathbf{M}_p)|_{X_m}).
$$
However, the above construction shows that
$$
(\mathbf{M}_k^*\otimes\mathbf{M}_p)|_{X_m}\simeq\otimes_iL_{im_i}^{\otimes a_i}
$$
with some $a_i$ negative. Therefore the vanishing part of property L yields
$$
H^1(X_m,(\mathbf{M}_k^*\otimes\mathbf{M}_p)|_{X_m})=0
$$
for $m\ge1$, and hence $\mathrm{Ext}^1(\mathbf{M}_k,\mathbf{M}_p)=0$
for $1\le p\le k\le s$. This is sufficient to conclude that the filtration
(\ref{*}) splits, i.e. (ii) follows.
\end{proof}

The rest of the paper (with exception of the appendix) is devoted to examples of linear ind-varieties satisfying
the properties L, A and T.

\end{sub}

\vspace{1cm}

\section{Linear ind-grassmannians satisfying the properties L,\ A,\ T}\label{sect4}
\vspace{0.5cm}

\begin{sub}{\bf Finite-dimensional orthogonal and symplectic grassmannians.}\label{subsection 2.1}

\rm
Let $V$ be a finite-dimensional vector space. In what follows we will consider, both symmetric and symplectic, quadratic forms $\Phi$ on $V$. Under the assumption that $\Phi$ is fixed, for any subspace $W\subset V$ we set $W^\perp:=\{v\in V\ |\ \Phi(v,w)=0$
for any $w\in W\}.$ Recall that $W$ is {\it isomorphic} (or $\Phi$-{\it isomorphic}) if $W\subset W^\perp.$

Let $\Phi\in S^2V^*$ be a non-degenerate symmetric form on $V$. For $\dim V\ge3$ and
$1\le k\le\left[\frac{\dim V}{2}\right]$, the {\it orthogonal grassmannian} $GO(k,V)$
is defined as the subvariety of
$G(k,V)$ consisting of all $\Phi$-isotropic $k$-dimensional subspaces of $V$. Unless
$\dim V=2n,$ $k=n,$ $GO(k,V)$
is a smooth irreducible variety. For
$\dim V=2n,$ $k=n,$ $GO(k,V)$
is smooth and has two irreducible components, both of which are isomorphic to
$GO(n-1,\tilde{V})$ for $\dim\tilde{V}=2n-1$.

If $\Phi\in\wedge^2V^*$ is a non-degenerate symplectic form on $V$, $\dim V=2n,$ we recall
that the {\it symplectic grassmannian} $GS(k,V)$ is a smooth irreducible subvariety of $G(k,V)$ consisting of all
$\Phi$-isotropic $k$-dimensional subspaces of $V$.

\end{sub}

\begin{sub}{\bf Definition of linear ind-grassmannians.}\label{subsection 4.2}

\rm
We start by recalling the definition of standard extension of grassmannians [PT3].

By a {\it standard extension of grassmannians} we understand an embedding of grassmannians
$f:G(k,V)\to G(k,V')$ for $\dim V\ge\dim V,$ $k'\ge k$, given by the formula
\begin{equation}\label{for21}
f:V_k\mapsto V_k\oplus W,
\end{equation}
for some fixed isomorphism $V'\simeq V\oplus \hat{W}$ and a fixed subspace $W\subset\hat{W}$ of dimension $k'-k.$
Respectively, by a {\it standard extension of orthogonal} respectively, of
{\it symplectic grassmannians} we
understand an embedding of isotropic grassmannians $f:GO(k,V)\to GO(k,V')$
(respectively, $f:GS(k,V)\to GS(k,V')$) given by the formula (\ref{for21}) for some fixed
orthogonal (respectively, symplectic) isomorphism
$V'\simeq V\oplus \hat{W}$ and a fixed isotropic subspace $W\subset\hat{W}$
of dimension $k'-k$, cf. \cite[Definitions 3.2. and 3.5]{PT3}. Note that standard extensions are linear morphisms.

Next we recall the definition of a standard ind-grassmannian [PT3].

\begin{definition}\label{G(k)}
Fix an infinite chain of vector spaces
\begin{equation}\label{for22}
V_{n_1}\subset V_{n_2}\subset...\subset V_{n_m}\subset V_{n_{m+1}}\subset...
\end{equation}
of dimensions $n_m,\ n_m< n_{m+1}$.

a) For an integer $k$, $1\le k< n_1$, set $\mathbf{G}(k):=\underset{\to}\lim G(k,V_{n_m})$ where
$$
G(k,V_{n_1})\hookrightarrow G(k,V_{n_2})\hookrightarrow...\hookrightarrow G(k,V_{n_m})\hookrightarrow
G(k,V_{n_{m+1}})\hookrightarrow...
$$
is the chain of canonical inclusions of grassmannians induced by (\ref{for22}).

b) For a sequence of integers $1\le k_1< k_2<...$ such that $ k_m< n_m,$  $\underset{m\to\infty}\lim(n_m-k_m)=\infty$,
set
$\mathbf{G}(\infty):=\underset{\to}\lim G(k_m,V_{n_m})$ where
$$
G(k_1,V_{n_1})\hookrightarrow G(k_2,V_{n_2})\hookrightarrow...\hookrightarrow G(k_m,V_{n_m})\hookrightarrow
G(k_{m+1},V_{n_{m+1}})\hookrightarrow...
$$
is an arbitrary chain of standard extensions of grassmannians.

c) Assume that $V_{n_m}$ are endowed with compatible non-degenerate symmetric (respectively, symplectic) forms $\Phi_m$.
In the symplectic case $\frac{n_m}{2}\in\mathbb{Z}_+$.
For an integer $k,\ 1\le k\le[\frac{n_1}{2}]$, set
 $\mathbf{G}\mathrm{O}(k,\infty):=\underset{\to}\lim GO(k,V_{n_m})$ (respectively,
 $\mathbf{G}\mathrm{S}(k,\infty):=\underset{\to}\lim GS(k,V_{n_m})$) where
$$
GO(k,V_{n_1})\hookrightarrow GO(k,V_{n_2})\hookrightarrow...\hookrightarrow GO(k,V_{n_m})\hookrightarrow
GO(k,V_{n_{m+1}})\hookrightarrow...
$$
(respectively,
$$
GS(k,V_{n_1})\hookrightarrow GS(k,V_{n_2})\hookrightarrow...\hookrightarrow GS(k,V_{n_m})\hookrightarrow
GS(k,V_{n_{m+1}})\hookrightarrow...)
$$
is the chain of inclusions of isotropic grassmannians induced by (\ref{for22}).

d) For a sequence of integers $1\le k_1< k_2<...$ such that $k_m<[\frac{n_m}{2}],$
$\underset{m\to\infty}\lim([\frac{n_m}{2}]-k_m)=\infty$,
set $\mathbf{G}\mathrm{O}(\infty,\infty)=\underset{\to}\lim G\mathrm{O}(k_m,V_{n_m})$
(respectively, $\mathbf{G}\mathrm{S}(\infty,\infty):=\underset{\to}\lim GS(k_m,V_{n_m})$)
where
\begin{equation}\label{chain O}
GO(k_1,V_{n_1})\hookrightarrow GO(k_2,V_{n_2})\hookrightarrow...\hookrightarrow GO(k_m,V_{n_m})\hookrightarrow
GO(k_{m+1},V_{n_{m+1}})\hookrightarrow...
\end{equation}
(respectively,
\begin{equation}\label{chain S}
GS(k_1,V_{n_1})\hookrightarrow GS(k_2,V_{n_2})\hookrightarrow...\hookrightarrow GS(k_m,V_{n_m})\hookrightarrow
GS(k_{m+1},V_{n_{m+1}})\hookrightarrow...)
\end{equation}
is an arbitrary chain of standard extensions of isotropic grassmannians.

e) In the symplectic case, consider a sequence of integers $1\le k_1< k_2<...$ such that $k_m\le\frac{n_m}{2}$,
$\underset{m\to\infty}\lim(\frac{n_m}{2}-k_m)=k\in\mathbb{N}$, and set
$\mathbf{G}\mathrm{S}(\infty,k):=\underset{\to}\lim GS(k_m,V_{n_m})$
for any chain of standard extensions (\ref{chain S}). In the orthogonal case, assume first that $\dim V_{n_m}$ are even.
Then set $\mathbf{G}\mathrm{O}^0(\infty,k):=\underset{\to}\lim GO(k_m,V_{n_m})$ for a chain (\ref{chain O}) where
$k_m<\frac{n_m}{2}$, $\underset{m\to\infty}\lim(\frac{n_m}{2}-k_m)=k\in\mathbb{N},\ k\ge2$. Finally, consider the
orthogonal case under the assumption that $\dim V_{n_m}$ are odd. Then set
$\mathbf{G}\mathrm{O}^1(\infty,k):=\underset{\to}\lim GO(k_m,V_{n_m})$ for a chain (\ref{chain O}) where
$k_m\le[\frac{n_m}{2}]$, $\underset{m\to\infty}\lim([\frac{n_m}{2}]-k_m)=k\in\mathbb{N}$.
\end{definition}

In particular,
$\mathbf{P}^{\infty}=\mathbf{G}(1)\simeq \mathbf{G}\mathrm{S}(1)$.
Note that the above standard ind-grassmannians are well-defined, i.e. a standard
ind-grassmannian does not depend, up to an isomorphism of ind-varieties, on the specific
chain of standard embeddings used in its definition. Furthermore, the main result of
[PT3] claims that, with the exception of the isomorphism
$\mathbf{P}^{\infty}\simeq \mathbf{G}\mathrm{S}(1)$,
the standard ind-grassmannians are pairwise non-isomorphic as ind-varieties.

In all cases the maximal exterior power of a tautological bundle generated by its global
sections yields an ample line bundle $\mathcal{O}_{X_m}(1)$, where
$X_m=G(k_m,V_{n_m}),$ $GO(k_m,V_{n_m}),$ $GS(k_m,V_{n_m})$.
It is well-known that $\mathcal{O}_{X_m}(1)$ generates
$\mathrm{\Pic}X_m$. Moreover, if
$i_m:X_m\hookrightarrow X_{m+1}$
is one of the embeddings in Definition \ref{G(k)}, there is an isomorphism
$i_m^*\mathcal{O}_{X_{m+1}}(1)\simeq\mathcal{O}_{X_m}(1)$.
This allows us to conclude that
$\mathbf{X}=\underset{\to}\lim X_m$ is a linear ind-variety and $\mathrm{\Pic}\mathbf{X}$ is generated by
$\mathcal{O}_{\mathbf{X}}(1):=\underset{\leftarrow}\lim\mathcal{O}_{X_m}(1)$.

\end{sub}

\begin{sub}{\bf BVTS Theorem for $\mathbf{G}(\infty),$
$\mathbf{G}\mathrm{O}(\infty,\infty)$, $\mathbf{G}\mathrm{S}(\infty,\infty)$,
$\mathbf{G}\mathrm{O}^1(\infty,0)$ and $\mathbf{G}\mathrm{S}(\infty,0)$.}\label{section 4.3}

\rm
We first note that, if $\mathbf{X}=\mathbf{G}(k)$, $\mathbf{G}\mathrm{O}(k,\infty)$,
$\mathbf{G}\mathrm{S}(k,\infty)$
there is a tautological rank-$k$
bundle $\mathbf{S}$ on $\mathbf{X}$. If $k\ge2$, this bundle is not isomorphic to a direct sum of line
bundles, hence the BVTS theorem does not hold for these ind-grassmannians.
Moreover, it is known [S2] that, for $\mathbf{X}=\mathbf{G}(k)$, $\mathbf{G}\mathrm{O}(k,\infty)$,
$\mathbf{X}=\mathbf{G}\mathrm{S}(k,\infty)$, any simple vector bundle of finite rank on
$\mathbf{X}$, i.e. a vector bundle which does not have a non-trivial proper subbundle,
is a direct summand in a tensor power of $\mathbf{S}$.

\begin{theorem}\label{theorem4.3}
Any vector bundle $ \mathbf{E}$ on $ \mathbf{X}\simeq \mathbf{G}(\infty),$ $\mathbf{G}\mathrm{O}(\infty,\infty)$, $\mathbf{G}\mathrm{S}(\infty,\infty)$, $\mathbf{G}\mathrm{O}^1(\infty,0),$ $\mathbf{G}\mathrm{S}(\infty,0)$
is isomorphic to
$\underset{i}\oplus\mathcal{O}_{\mathbf{X}}(k_i)$
for some $k_i\in\mathbb{Z}$.
\end{theorem}

For $ \mathbf{X}=\mathbf{G}(\infty)$ this is proved in [DP] (see also [PT1, Section 4]). For the remaining standard
ind-grassmannians the claim of Theorem~\ref{theorem4.3} follows from Theorem \ref{E on linear X} and the following
theorem.

\begin{theorem}\label{theorem4.4}
Let $ \mathbf{X}\simeq \mathbf{G}\mathrm{O}(\infty,\infty)$,
$\mathbf{G}\mathrm{S}(\infty,\infty)$, $\mathbf{G}\mathrm{O}^1(\infty,0)$ or
$\mathbf{G}\mathrm{S}(\infty,0).$ Then $\mathbf{X}$ satisfies the properties {\rm L, A} 
and {\rm T}.\footnote{The reader can check that $\mathbf{G}(\infty)$ also satisfies the properties L, A, T.}

\end{theorem}

\begin{proof}
$\mathbf{X}$ satisfies the property L as $\mathcal{O}_{\mathbf{X}}(1)$ generates
$\Pic \mathbf{X}$, and $H^1(X_m,\mathcal{O}_{X_m}(a))$ vanishes for all $a$ and sufficiently large $m$
by Borel-Weil-Bott's Theorem. Furthermore, the property T follows from
Proposition \ref{proposition7.4} below.

It remains to establish the property A. Part (A.i) holds here simply because
$\mathcal{O}_{\mathbf{X}}(1)$ is very ample. We therefore discuss parts (A.ii)-(A.iv).

Let $\mathbf{X}=\mathbf{G}\mathrm{O}(\infty,\infty)=\underset{\to}\lim GO(k_m,V_{n_m})$.
For $m\ge1,$ the base $B_m$ of the family of projective lines on $GO(k_m,V_{n_m})$ coincides with the variety of isotropic flags of type
$(k_m-1,k_m+1)$ in $V_{n_m}$:
\begin{equation}\label{FlO}
B_m=\{(V_{k_m-1},V_{k_m+1})\in GO(k_m-1,V_{n_m})\times GO(k_m+1,V_{n_m})\ |\ V_{k_m-1}\subset V_{k_m+1}\}
\end{equation}
[PT3, Lemma 2.2(i)]. Furthermore, set
\begin{equation}\label{Pim}
\Pi_m:=\{(V_{k_m},V_{k_m+1})\in GO(k_m,V_{n_m})\times GO(k_m+1,V_{n_m})\ |\ V_{k_m}\subset V_{k_m+1}\}.
\end{equation}
A point $y=(V_{k_m},V_{k_m+1})\in\Pi_m$ corresponds to the projective subspace $G(k_m-1,V_{n_m})\times\{V_{k_m+1}\}\subset B_m.$

It is easy to see that $\mathbf{\Pi}:=\underset{\to}\lim \Pi_m$ is a well-defined ind-variety and that a point of
$\mathbf{\Pi}$ represents a projective ind-subspace of $\mathbf{B}:=\underset{\to}\lim B_m.$

Next, (\ref{FlO}) together with [PT3, Lemma 2.2(iv)], implies that for any point
$x=\{V_{k_m}\}\in GO(k_m,V_{n_m})$,
\begin{equation}\label{Bm(x)}
B_{\tilde m}(x):=\{\mathbb{P}^1\in B_{\tilde m}\ |\ \mathbb{P}^1\ni x\}\simeq
\end{equation}
$$
\mathbb{P}((\phi_{{\tilde m}-1}\circ...\circ\phi_m)(V_{k_m})^*)\times GO(1,(\phi_{{\tilde m}-1}\circ...\circ\phi_m)(V_{k_m})^\perp/(\phi_{{\tilde m}-1}\circ...\circ\phi_m)(V_{k_m})),\ \ \tilde{m}\ge m,
$$
and
\begin{equation}\label{Pim(x)}
\Pi_{{\tilde m}}(x):=\{((\phi_{{\tilde m}-1}\circ...\circ\phi_m)(V_{k_m}),
V_{k_{\tilde m+1}})\in\Pi_{{\tilde m}}\}\simeq
GO(1,V_{k_{\tilde m}}^\perp/V_{k_{\tilde m}}),\ \ \tilde{m}\ge m.
\end{equation}

Since the quadrics
$GO(1,(\phi_{{\tilde m}-1}\circ...\circ\phi_m)(V_{k_m})^\perp/(\phi_{{\tilde m}-1}\circ...\circ\phi_m)(V_{k_m}))$ are connected, (A.ii) follows from (\ref{FlO}). Furthermore, as for
each $\tilde{m}\ge1$ the variety $GO(1,V_{k_{\tilde m}}^\perp/V_{k_{\tilde m}})$ is a
smooth quadric hypersurface in the projective space
$\mathbb{P}^{n_{\tilde m}-2k_{\tilde m}-1}$,
(\ref{Bm(x)}) and (\ref{Pim(x)}) directly imply (A.iii) and (A.iv).

In the remaining cases the same argument goes through if one makes the following
modifications.

If $ \mathbf{X}=\mathbf{G}\mathrm{S}(\infty,\infty)$, the formulas for $B_m$, $\Pi_m$ and
$\Pi_{{\tilde m}}(x)$ are the same as (\ref{FlO}), (\ref{Pim}), (\ref{Bm(x)}) and
(\ref{Pim(x)}) respectively, with $GO$ substituted by $GS$ (use [PT3, Lemma 2.5]).
Note also that $GS(1,V_{k_{\tilde m}}^\perp/V_{k_{\tilde m}})$
is isomorphic to the projective space
$\mathbb{P}(V_{k_{\tilde m}}^\perp/V_{k_{\tilde m}})$.

For
$\mathbf{X}=\mathbf{G}\mathrm{O}^1(\infty,0)=\underset{\to}\lim GO(k_m,V_{2k_m+1})$
we first identify
$GO(k_m,V_{2k_m+1})$
with an irreducible component
$GO(k_m+1,V_{2k_m+2})^*$ of $GO(k_m+1,V_{2k_m+2})$ - see \cite[Section 2.3]{PT3}.
Consequently,
$ \mathbf{X}\simeq\underset{\to}\lim GO(k_m,V_{2k_m})^*$.
Next, instead of (\ref{FlO})-(\ref{Pim}) one has
$B_m\simeq GO(k_m-2,V_{2k_m}),$ $\Pi_m\simeq GO(k_m-1,V_{2k_m}),$ $m\ge1.$
Respectively, instead of (\ref{Bm(x)})-(\ref{Pim(x)}) one has
$B_{\tilde m}(x)\simeq G(k_{\tilde m}-2,
(\phi_{{\tilde m}-1}\circ...\circ\phi_m)(V_{k_m}))$ for $x=\{V_{k_m}\}$.
The latter fact can be proved by an argument similar to that of [PT3, Lemma 2.2].
In addition,
$\Pi_{{\tilde m}}(x)\simeq \mathbb{P}((\phi_{{\tilde m}-1}\circ...\circ\phi_m)(V_{k_m})^*),$ $\tilde{m}\ge m$.

For $ \mathbf{X}=\mathbf{G}\mathrm{S}(\infty,0)=\underset{\to}\lim GS(k_m,V_{2k_m})$
one can show that (\ref{FlO})-(\ref{Pim}) can be replaced by:
$B_m\simeq GS(k_m-1,V_{2k_m}),$ $\Pi_m\simeq GS(k_m,V_{2k_m})$.
Respectively, (\ref{Bm(x)})-(\ref{Pim(x)}) for $x=\{V_{k_m}\}\in GS(k_m,V_{2k_m})$
can be replaced by
$B_{\tilde m}(x)\simeq G(k_{\tilde m}-1,(\phi_{{\tilde m}-1}\circ...\circ\phi_m)(V_{k_m}))$,
and $\Pi_{{\tilde m}}(x)\simeq \{\mathbb{P}((\phi_{{\tilde m}-1}\circ...\circ\phi_m)(V_{k_m})^*)\}$ is a point
for $\tilde{m}\ge m$.
\end{proof}

\end{sub}

\vspace{1cm}

\section{Linear sections of $\mathbf{G}(\infty)$, $\mathbf{G}\mathrm{O}(\infty,\infty)$,
$\mathbf{G}\mathrm{S}(\infty,\infty)$}\label{sec5}

\vspace{0.5cm}

\begin{sub}{\bf Linear sections of finite-dimensional grassmannians.}\label{5.1}

\rm
Let $G=G(k,V)$, $GO(k,V)$, $GS(k,V)$.
Assume $1\le k<\dim V-1$ for $G=G(k,V)$, and $1\le k<[\frac{\dim V}{2}]$ for $G=GO(k,V)$, $GS(k,V)$.
Put
$N:=\dim H^0(\mathcal{O}_G(1))$ and
$V_N:=H^0(\mathcal{O}_G(1))^*.$
We consider $G$ as a subvariety of $\mathbb{P}(V_N)$
via the Pl\"ucker embedding. For a given integer $c$, $1\le c\le k-1$, set
$$
X:=G\cap\mathbb{P}(U),
$$
where $U\subset V_N$ is a subspace of codimension $c$. We call $X$ a {\it linear section of $G$ of codimension $c$.}

Note that there is a single family of maximal projective spaces of dimension $k$ on
$G$ with base $\tilde{G},$ where
$\tilde{G}=G(k+1,V)$ if $G=G(k,V)$, respectively, $\tilde{G}=GO(k+1,V)$ if $G=GO(k,V)$, and
$\tilde{G}=GS(k+1,V)$ if $G=GS(k,V)$ (see [PT3, Lemmas 2.2(i) and 2.5(i)]). Consider the graph of incidence
$\Sigma:=\{(V_k,V_{k+1})\in G\times\tilde{G}\ |\ V_k\subset V_{k+1}\}$ with projections
$\tilde{G}\overset{p}\leftarrow\Sigma\overset{q}\to G$ and  set
$\pi:=p|_{q^{-1}(X)}:q^{-1}(X)\rightarrow \tilde{G}.$ The condition $1\le c\le k-1$ implies that
$\pi$ is a surjective projective morphism.

\begin{proposition}\label{proposition5.1}
For a subspace $U\subset V_N$ of codimension $c$ in general
position the following statements hold.

(i) The varieties $X$ and $q^{-1}(X)$ are smooth and
\begin{equation}\label{equation30}
\pi_*\mathcal{O}_{q^{-1}(X)}=\mathcal{O}_{\tilde{G}}.
\end{equation}

(ii) $Z(U):=\{x\in \tilde{G}\ |\ \dim \pi^{-1}(x)>k-c\}$ is a proper closed subset of $\tilde{G}$ and
\begin{equation}\label{equation31}
{\rm codim}_{\tilde{G}}Z(U)\ge3,\ \ \ {\rm codim}_{q^{-1}(X)}\pi^{-1}(Z(U))\ge2.
\end{equation}

(iii) The projection $\pi:q^{-1}(X)\setminus\pi^{-1}(Z(U))\rightarrow \tilde{G}\setminus Z(U)$ is a projective $\mathbb{P}^{k-c}$-bundle.

\end{proposition}

\begin{proof}
We give the proof for the case $G=GO(k,V)$. The other cases are very similar and we leave them to the reader.

(i) Since the projective subspace
$\mathbb{P}(U)$ is in general position in $\mathbb{P}(V_N)$, we have
${\rm codim}_G X=c$ and hence there is a Koszul resolution of the
$\mathcal{O}_G$-sheaf $\mathcal{O}_X$
\begin{equation}\label{equation32}
0\rightarrow \mathcal{O}_G(-c)\rightarrow ...\rightarrow\left(c\atop i\right)
\mathcal{O}_G(-i)\rightarrow ...\rightarrow c\mathcal{O}_G(-i)\rightarrow
\mathcal{O}_G\rightarrow \mathcal{O}_X\rightarrow 0.
\end{equation}
The pullback of (\ref{equation32}) under the projection $q$ is a $\mathcal{O}_\Sigma$-resolution of the
sheaf $\mathcal{O}_{q^{-1}(X)}=\pi^*\mathcal{O}_{\tilde{G}}$ of the form
\begin{equation}\label{equation33}
0\rightarrow \mathcal{L}_c\rightarrow ...\rightarrow\mathcal{L}_1\rightarrow
\mathcal{O}_{\Sigma}\rightarrow \pi^*\mathcal{O}_{\tilde{G}} \rightarrow 0
\end{equation}
where $\mathcal{L}_i:=q^*(\left(c\atop i\right)\mathcal{O}_G(-i))$, $i=1,...,c$.

For any $x\in \tilde{G}$ we have $p^{-1}(x)\simeq \mathbb{P}^k$, so the condition $c\le k-1$ implies
$H^j(p^{-1}(x),\mathcal{L}_i|_{p^{-1}(x)})\simeq
H^j(\mathbb{P}^k,\mathcal{O}_{\mathbb{P}^k}(-i))=0$ for $j\ge0$, $i=1,...,c.$
Hence the Base-change Theorem \cite[Ch. III, Theorem 12.11]{H1} for the flat projective morphism
$p$ shows that $R^jp_*\mathcal{L}_i=0$.
In addition, by the same reason $R^jp_*\mathcal{O}_\Sigma=0,$ $j>0$,
and clearly $p_*\mathcal{O}_{\Sigma}=\mathcal{O}_{\tilde{G}}.$
Therefore, applying the functor $R^\cdot p_*$ to (\ref{equation33}) we obtain
(\ref{equation30}).

(ii) We now prove (\ref{equation31}). Fix an arbitrary point $V_{k+1}\in \tilde{G}.$ Since $p^{-1}(x)=\mathbb{P}(V^*_{k+1})$
(see [PT3, Lemma 2.2(i)]), there is an induced monomorphism
\begin{equation}\label{equation34}
0\rightarrow V^*_{k+1}\rightarrow V_N.
\end{equation}
Consider the varieties
$$
\Gamma_i:=\{(W,V_{k+1})\in G(N-c, V_N)\times \tilde{G}\ |\ \dim(W\cap V^*_{k+1})\ge k-c+i+2\},\ \ 0\le i\le c-1,
$$
together with the natural projections
$$
G(N-c, V_N)\overset{p_i}\leftarrow\Gamma_i\overset{q_i}\to \tilde{G}.
$$
For an arbitrary $U\in G(N-c, V_N)$ denote
$$
Z_i(U):=q_i(p_i^{-1}(U)),
\ \ 0\le i\le c-1.
$$

By construction, $Z_0(U)=Z(U)$, $Z_i(U)$ are closed subvarieties of $\tilde{G}$,
and we have a filtration
\begin{equation}\label{equation38}
\emptyset=:Z_c(U)\subset Z_{c-1}(U)\subset ...\subset Z_{0}(U)=Z(U)
\end{equation}
such that $Z_i(U)':=Z_i(U)\setminus Z_{i+1}(U)$) are locally closed subvarieties of $\tilde{G}$.
Consequently, $B_i(U)':=\pi^{-1}(Z_i(U)')$ are locally closed subvarieties of $q^{-1}(X)$.
Moreover, $\pi|_{B_i(U)'}:B_i(U)'\rightarrow Z_i(U)'$ is a $\mathbb{P}^{k+1-c+i}$-bundle, so that
$\dim B_i(U)'=\dim Z_i(U)'+k+1-c+i$.
Equivalently,
\begin{equation}\label{equation39}
{\rm codim}_{q^{-1}(X)}B_i(U)'={\rm codim}_{\tilde{G}}Z_i(U)'-(i+1).
\end{equation}
Note also that $Z(U)=\cup_{i=0}^{c-1}Z_i(U)'$, hence
\begin{equation}\label{equation40}
\pi^{-1}(Z(U))=\pi^{-1}\left( \cup_{i=0}^{c-1}Z_i(U)'\right)=\cup_{i=0}^{c-1}B_i(U)'.
\end{equation}

We now calculate the dimensions of $Z_i(U)$ under the assumption that $U$ is in general positon.
For this, let $Y:=q_i^{-1}(x)$
be the fibre of the projection $q_i$ over a point $x=V_{k+1}\in Z_i(U)$. Consider the variety
$\tilde{Y}=\{(W,V_{k-c+i+2})\in G(N-c, V_N)\times G(k-c+i+2, V_{k+1}^*)\ |\
W\supset V_{k-c+i+2}\subset V_{k+1}^*\}$.
The natural projection
$\tilde{Y}\rightarrow G(k-c+i+2,V_{k+1}^*)$
is a fibration with the grassmannian $G(N-k-i-2,\mathbb{C}^{N-k-i-2+c})$
as a fibre. On the other hand, one has a birational surjective morphism
$\tilde{Y}\to Y,\ (W,V_{k-c+i+2})\mapsto W$.
Therefore, in view of (\ref{equation34}),
$\dim Y=\dim \tilde{Y}=\dim G(k-c+i+2, V_{k+1}^*)+
\dim G(N-k-i-2,\mathbb{C}^{N-k-i-2+c})=cN-c^2+(i+1)(c-k-i-2)$.
As $q_i$ is surjective, this yields
$$
\dim\Gamma_i=\dim \tilde{G}+\dim Y=\dim \tilde{G}+cN-c^2+(i+1)(c-k-i-2).
$$
Since $p_i$ is also surjective, for a point $U\in G(N-c,V_N)$ in general position we have
$\dim Z_i(U)=\dim\Gamma_i-\dim G(N-c, V_N)=\dim \tilde{G}-(i+1)(k+i+2-c),$ i.e.
\begin{equation}\label{equation42}
{\rm codim}_{\tilde{G}}Z_i(U)=(i+1)(k+i+2-c),\ \ 0\le i\le c-1.
\end{equation}
This together with (\ref{equation39}) implies
${\rm codim}_{q^{-1}(X)}B_i(U)'=(i+1)(k+i+1-c),$ $0\le i\le c-1.$
Therefore, in view of (\ref{equation40}) and the assumption $c\le k-1$, we obtain
\begin{equation}\label{equation43}
{\rm codim}_{q^{-1}(X)}\pi^{-1}(Z(U))=
\underset{0\le i\le c-1}{\rm min}{\rm codim}_{q^{-1}(X)}B_i(U)'=k+1-c\ge2.
\end{equation}
The inequality $\codim_{\tilde{G}}Z(U)\ge3$ follows now from (\ref{equation39}), and the proposition
is proved.
\end{proof}

\begin{corollary}\label{corollary5.2}
Under the assumptions of Proposition \ref{proposition5.1}, let $\mathcal{E}$
be a vector bundle on $q^{-1}(X)$ trivial along the fibres of the morphism
$\pi:q^{-1}(X)\to \tilde{G}$. Then there is an isomorphism
$ev:\pi^*\pi_*\mathcal{E} \overset{\simeq}\to \mathcal{E}$.
\end{corollary}

\begin{proof}
Apply Proposition \ref{proposition7.2} from the appendix to the
morphism $\pi:q^{-1}(X)\to\tilde{G}$,
the subvariety $Z(U)$ in $\tilde{G}$,
and the vector bundle $\mathcal{E}$ on $q^{-1}(X)$.
\end{proof}

\begin{lemma}\label{lemma5.3}
Let $X=G\cap\mathbb{P}(U)$ be a linear section of $G$ of codimension $c$ for
$1\le c\le(k-1)/2,$ and let $\mathbb{P}^1$ be a projective line on $\tilde{G}$.
Then there exists a rational curve
$C\subset q^{-1}(X)$ such that $\pi|_C$ is an isomorphism of $C$ with $\mathbb{P}^1$, and
$q|_C$ is either an isomorphism or a constant map.
\end{lemma}

\begin{proof}
We only consider the case $G=GO(k,V)$. It is clear that
$$
\mathbb{P}^1=\{V_{k+1}\in \tilde{G}\ |\ V_k\subset V_{k+1}\subset V_{k+2}\}
$$
for a unique isotropic flag $V_k\subset V_{k+2}$ in $V$. If $V_k\in X,$ we set
$C:=\{(V_k,V_{k+1})\in\Sigma\ |\ V_k\subset V_{k+1}\subset V_{k+2}\}$.
Then $\pi|_C:C\to\mathbb{P}^1$ is an isomorphism and $q(C)$ equals the point $\{V_k\}\in G$.

Assume that $V_k\notin X.$ It is straightforward to check that the intersection
$q(p^{-1}(\mathbb{P}^1))\cap \mathbb{P}^{N-2}$ for a hyperplane
$\mathbb{P}^{N-2}\subset \mathbb{P}(V_N)$ such that $V_k\notin \mathbb{P}^{N-2}$,
is isomorphic to the direct product
$\mathbb{P}(V_k^*)\times \mathbb{P}(V_{k+2}/V_k)$
imbedded by Segre in $\mathbb{P}^{N-2}$. Let
$\mathbb{P}^{k-1}_a$, $\mathbb{P}^{k-1}_b$
be the fibres in
$\mathbb{P}(V_k^*)\times \mathbb{P}(V_{k+2}/V_k)$
over two points
$a,b\in  \mathbb{P}(V_{k+2}/V_k).$
The projection $pr_1:\mathbb{P}(V_k^*)\times \mathbb{P}(V_{k+2}/V_k)\to\mathbb{P}(V_k^*)$ induces an isomorphism
$f:\mathbb{P}^{k-1}_a\overset{\sim}\to\mathbb{P}^{k-1}_b.$

Set
$\mathbb{P}^{k-c-1}_a:=\mathbb{P}^{k-1}_a\cap \mathbb{P}(U)$,
$\mathbb{P}^{k-c-1}_b:=\mathbb{P}^{k-1}_b\cap
\mathbb{P}(U)$. Since $1\le c\le(k-1)/2$, the intersection
$\mathbb{P}^{k-c-1}_b\cap f(\mathbb{P}^{k-c-1}_a)\subset\mathbb{P}^{k-1}_b$
is nonempty. Consider a point $x$ in this latter intersection. By construction, the fibre
$\mathbb{P}^1_x:=pr_1^{-1}(x)$ lies in
$q(p^{-1}(\mathbb{P}^1))\cap X$.
Finally, the preimage of $\mathbb{P}^1_x$ in
$p^{-1}(\mathbb{P}^1)$ is a rational curve $C$ as desired.
\end{proof}

\begin{proposition}\label{proposition5.4}
Let $X$ be a linear section of $G$ of codimension $c$ for $1\le c\le(k-1)/2.$ Then a linearly trivial vector bundle $E$ on $X$ is trivial.
\end{proposition}

\begin{proof}
Consider the vector bundle $\mathcal{E}:=q^*E$ on $q^{-1}(X)$.
Since $E$ is linearly trivial, for any $x\in \tilde{G}$ $\mathcal{E}|_{\pi^{-1}(x)}$
is a linearly trivial bundle on the projective space $\pi^{-1}(x)$.
A  well-known theorem \cite[Ch. I, Theorem 3.2.1]{OSS}
implies that $\mathcal{E}|_{\pi^{-1}(x)}$ is trivial. Therefore
$ev:\pi^*\pi_*\mathcal{E}\to\mathcal{E}$
is an isomorphism by Corollary \ref{corollary5.2}.

Next, Lemma \ref{lemma5.3} allows us to conclude that $\pi_*\mathcal{E}$ is linearly
trivial. Indeed, if $\mathbb{P}^1\subset \tilde{G}$ is a projective line and
$C\subset q^{-1}(X)$ is a rational curve as in Lemma \ref{lemma5.3}, then
$\pi_*\mathcal{E}|_{\mathbb{P}^1}\simeq \mathcal{E}|_C$,
and hence $\mathcal{E}|_C$ is trivial because of the linear triviality of $E$.

Consequently, $\pi_*\mathcal{E}$ is trivial by Proposition 7.4. from the appendix.
Then $\mathcal{E}\simeq\pi^*\pi_*\mathcal{E}$ is also trivial.
Finally, since $q: q^{-1}(X)\to X$ is a flat projective morphism with irreducible fibres,
$E=q_*\mathcal{E}$ is trivial by Proposition \ref{trivial on fibres}.
\end{proof}

\end{sub}

\begin{sub}{\bf Linear sections of
$\mathbf{G}(\infty),$ $\mathbf{G}\mathrm{O}(\infty,\infty)$,
$\mathbf{G}\mathrm{S}(\infty,\infty)$ of small codimension.}
\rm

Let $\mathbf{G}=\mathbf{G}(\infty),$ $\mathbf{G}\mathrm{O}(\infty,\infty)$, $\mathbf{G}\mathrm{S}(\infty,\infty)$, in particular $ \mathbf{G}=\underset{\to}\lim G(k_m,V_{n_m})$,
$\underset{\to}\lim GO(k_m,V_{n_m}),$ $\underset{\to}\lim GS(k_m,V_{n_m})$, see Definition
\ref{G(k)}. Fix a nondecreasing sequence $\{c_m\}_{m\ge1}$ of integers satisfying the condition
\begin{equation}\label{equation45}
1\le c_m\le(k_m-1)/2.
\end{equation}
Consider an ind-variety $\mathbf{X}=\underset{\to}\lim X_m$ such that,
for each $m\ge1$, $X_m$ is a smooth linear section of codimension $c_m$ of $G_m=G(k_m,V_{n_m})$, $GO(k_m,V_{n_m})$, $GS(k_m,V_{n_m})$ and the embedding
$\phi_m:X_m\hookrightarrow X_{m+1}$ is induced by the embedding
$G_m\hookrightarrow G_{m+1}$. In what follows we call such ind-varieties {\it linear
sections of} $\mathbf{G}$ {\it of small codimension.}

The existence of linear sections $\mathbf{X}$ of $\mathbf{G}$ of small codimension is a consequence of
the Bertini Theorem. Moreover, such a linear section $\mathbf{X}$ is a linear ind-variety and
$\Pic\mathbf{X}$ is generated by
$\mathcal{O}_{\mathbf{X}}(1):=\underset{\leftarrow}\lim \mathcal{O}_{X_m}(1)$.
This follows from the observation that $\mathcal{O}_{X_m}(1)$ generates
$\Pic X_m$ by the Lefschetz Theorem, and from the linearity of the embeddings
$G_m\hookrightarrow G_{m+1}.$

\begin{proposition}\label{theorem5.5}
Let $\mathbf{G}=\mathbf{G}(\infty)$, $\mathbf{G}\mathrm{O}(\infty,\infty)$, $\mathbf{G}\mathrm{S}(\infty,\infty)$.
There exists a linear section $\mathbf{X}$ of small codimension of $\mathbf{G}$ which is
not isomorphic to either of the ind-varieties
$\mathbf{G}(\infty)$, $\mathbf{G}\mathrm{O}(\infty,\infty)$ or
$\mathbf{G}\mathrm{S}(\infty,\infty)$.
\end{proposition}

\begin{proof}
Let $\mathbf{G}=\underset{\to}\lim G_m$ where
$G_m=G(k_m,V_{n_m})$, $GO(k_m,V_{2n_m+1})$, $GS(k_m,V_{2n_m})$. Fix $m\ge1$ and let
$\mathbb{P}^1$ be a projective line on $G_m$.
Then there exist unique maximal projective subspaces
$\mathbb{P}^{k_m}$ and $\mathbb{P}^{s_m}$ on
$G_m$ which intersect in $\mathbb{P}^1$.
For $G_m=G(k_m,V_{n_m})$ one has $s_m=n_m-k_m$, and for $G_m=GO(k_m,V_{n_m})$, $GS(k_m,V_{n_m})$
one has $s_m=[\frac{n_m}{2}]-k_m$, see \cite[Lemmas 2.3(iii) and 2.6(ii)]{PT3}.

For $\tilde{m}\ge m$ the embeddings
$\phi_{\tilde{m}}:G_{\tilde{m}}\hookrightarrow G_{\tilde{m}+1}$ in the direct limit
$\underset{\to}\lim G_m$
are given by formula (\ref{for21}), which makes it easy to check that
$\mathbb{P}^{k_m}$ and $\mathbb{P}^{s_m}$ admit extensions
$\mathbb{P}^{k_{\tilde{m}}}\subset G_{\tilde{m}}$,
$\mathbb{P}^{s_{\tilde{m}}}\subset G_{\tilde{m}}$
such that
$\mathbb{P}^{k_m}\subset\mathbb{P}^{k_{\tilde{m}}}$,
$\mathbb{P}^{s_m}\subset\mathbb{P}^{s_{\tilde{m}}}$
and
$$
\mathbb{P}^1=\mathbb{P}^{k_{\tilde{m}}}\cap\mathbb{P}^{s_{\tilde{m}}},\ \ {\tilde{m}}\ge m.
$$
Denoting
$\mathbf{P}^\infty_\alpha:=\underset{\to}\lim \mathbb{P}^{k_{\tilde{m}}}$,
$\mathbf{P}^\infty_\beta:=\underset{\to}\lim \mathbb{P}^{s_{\tilde{m}}}$,
we have
\begin{equation}\label{equation50}
\mathbb{P}^1=\mathbf{P}^\infty_\alpha\cap\mathbf{P}^\infty_\beta.
\end{equation}

We now choose $n_{\tilde{m}}$ and $k_{\tilde{m}}$ in a specific way.
Namely, we assume that $n_{\tilde{m}}=3t_{\tilde{m}}$ for $G_m=G(k_m,V_{n_m})$,
and $[\frac{n_{\tilde{m}}}{2}]=3t_{\tilde{m}}$ for
$G_m=GO(k_m,V_{n_m})$, $GS(k_m,V_{n_m})$,\ $k_{\tilde{m}}=2t_{\tilde{m}}$, $t_{\tilde{m}}\in\mathbb{Z}_{\ge1}$.
Set $c_{\tilde{m}}:=t_{\tilde{m}}-1$.
Then
\begin{equation}\label{=1}
s_{\tilde{m}}-c_{\tilde{m}}=1
\end{equation}
and the inequality (\ref{equation45}) together with the conditions
$\underset{{\tilde{m}}\to\infty}\lim k_{\tilde{m}}=
\underset{{\tilde{m}}\to\infty}\lim s_{\tilde{m}}=\infty$
are satisfied. Next, using (\ref{=1}) and the Bertini Theorem we choose a tower of projective subspaces
$\mathbb{P}^{N_{\tilde{m}}-c_{\tilde{m}}-1}\subset\mathbb{P}(V_{N_{\tilde{m}}})$ for ${\tilde{m}}\ge m$
in general positon so that
$\mathbb{P}^{N_{\tilde{m}}-c_{\tilde{m}}-1}\cap\mathbb{P}^{s_{\tilde{m}}}=\mathbb{P}^1$
and
$\mathbb{P}^{N_{\tilde{m}}-c_{\tilde{m}}-1}\cap\mathbb{P}^{k_{\tilde{m}}}=
\mathbb{P}^{(k_{\tilde{m}}/2)+1}$
for some projective subspaces
$\mathbb{P}^{(k_{\tilde{m}}/2)+1}$ of $G_{{\tilde{m}}}$.
As a result, we obtain a linear section
$\mathbf{X}:=\underset{\to}\lim(\mathbb{P}^{N_{\tilde{m}}-c_{\tilde{m}}-1}\cap G_{{\tilde{m}}})$
of $\mathbf{G}$ of small codimension and a projective line $\mathbb{P}^1\subset\mathbf{X}$ such that
$\mathbb{P}^1$ is contained in a unique linear ind-projective subspace $\mathbf{P}^\infty$
of $\mathbf{X}$, namely
$\mathbf{P}^\infty:=\underset{\to}\lim \mathbb{P}^{(k_{\tilde{m}}/2)+1}$.
If $\mathbf{X}$ were isomorphic to
$\mathbf{G}(\infty)$, $\mathbf{G}\mathrm{O}(\infty,\infty)$,
$\mathbf{G}\mathrm{S}(\infty,\infty)$, this would contradict to (\ref{equation50}).

\end{proof}

Now we show that a linear section $\mathbf{X}$ of $\mathbf{G}$ of small codimension satisfies the
properties L, A and T. The property L is clear as $\Pic\mathbf{X}$ is generated by
$\mathcal{O}_{\mathbf{X}}(1)$, and $H^1(X_m,\mathcal{O}_{X_m}(a))$ vanishes for $a<0$ by
Kodaira's Theorem. The property T is established in Proposition \ref{proposition5.4}.
It remains to establish the property A. Part (A.i) is clear as $\mathcal{O}_{X_m}(1)$
is very ample. For parts (A.ii)-(A.iv) we consider in detail only the case when
$\mathbf{X}$ is a linear section of $\mathbf{G}\mathrm{O}(\infty,\infty)$.

Let $B(G_m)$ be the family of all projective lines in $G_m=GO(k_m,V_{n_m})$, and $B_m$ be its subfamily consisting of
those projective lines which lie in $X_m$. By definition, $X_m$ is the intersection of $G_m$ with a subspace
$\mathbb{P}(U_m)$ of $\mathbb{P}(V_{N_m})$ for a fixed $U_m\in G(N_m-c_m,V_{N_m})$ in general position.
The grassmannians $G(2,V_{N_m})$ and $G(2,U_m)$
can be be thought of as the grassmannians of projective lines in
$\mathbb{P}(V_{N_m})$ and $\mathbb{P}(U_m)$ respectively. Then $B_m=B(G_m)\cap G(2,U_m)$ where the intersection is taken in $G(2,V_{N_m})$. We show next that $B_m$ is irreducible.

Let $B$ be an irreducible component of $B_m$. Since $G(2,V_{N_m})$ is smooth, the
subadditivity of  codimensions \cite[Thm. 17.24]{Ha}) yields
\begin{equation}\label{equation46}
{\rm codim}_{B(G_m)}B\le{\rm codim}_{G(2,V_{N_m})}G(2,U_{m})=2c_m.
\end{equation}
Consider the graph of incidence
$\Sigma_m:=\{(x,\mathbb{P}^{k_m})\in G_m\times \tilde{G}_m\ |\ x\in\mathbb{P}^{k_m}\}$
with its projections
$\tilde{G}_m\overset{p_m}\leftarrow\Sigma_m\overset{q_m}\to G_m$,
where $\tilde{G}_m:=GO(k_m+1,V_{n_m})$. Let $B(\Sigma_m)$ be the family of all projective lines in $\Sigma_m$ lying in the fibres of the projection $p_m$.
Denote by $B(q_m^{-1}(X_m))$ the subfamily of $B(\Sigma_m)$ consisting of those projective
lines which lie in $q_m^{-1}(X_m)$. The projection $q_m:\Sigma_m\to G_m$ induces a morphism $r_{G_m}:B(\Sigma_m)\to B(G_m)$ which is bijective since any projective line on $G_m$ lies
in a unique maximal projective space $\mathbb{P}^{k_m}$
(see \cite[Section 2]{PT3}). The space $\mathbb{P}^{k_m}$ is an isomorphic image via
$q_m$ of some fibre of $p_m$. Respectively, the restricted morphism
$r_{X_m}:=r_{G_m}|_{B(q_m^{-1}(X_m))}:B(q_m^{-1}(X_m))\to B_m$
is a bijection. Hence, for any irreducible component $B'$ of $B(q_m^{-1}(X_m))$,
(\ref{equation46}) yields the inequality
\begin{equation}\label{equation47}
{\rm codim}_{B(\Sigma_m)}B'\le{\rm codim}_{G(2,V_{N_m})}G(2,U_{m})=2c_m.
\end{equation}

The projection $p_m$ induces a projection $\rho_m:B(\Sigma_m)\to \tilde{G}_m$. Let
$$
\emptyset=Z_{c_m}(U_m)\subset Z_{c_m-1}(U_m)\subset...\subset Z_0(U_m)\subset \tilde{G}_m
$$ be the filtration (\ref{equation38}) of $\tilde{G}_m$ by closed subvarieties
$Z_i(U_m)$ of codimensions in $\tilde{G}_m$ given by (\ref{equation42}) where we put $c=c_m,$ $k=k_m$.
This filtration yields a decomposition $B(q_m^{-1}(X_m))=\underset{0\le i\le c_m}\sqcup B_i$,
where
$B_0:=B(q_m^{-1}(X_m))\cap\rho^{-1}_m(\tilde{G}_m\setminus Z_0(U_m))$,
$B_i:=B(q_m^{-1}(X_m))\cap\rho^{-1}_m(Z_{i-1}(U_m)\setminus Z_i(U_m))$, $i=1,...,c_m$. Formula (\ref{equation42})
implies that ${\rm codim}_{B(\Sigma_m)}B_0=2c_m$, ${\rm codim}_{B(\Sigma_m)}B_i>2c_m$ for $i=1,...,c_m$. This together
with (\ref{equation47}) yields the irreducibility of $B(q_m^{-1}(X_m))$, hence of $B_m$ as well.

Now let $\Pi(G_m)$ be the family of projective spaces $\mathbb{P}^{k_m-1}$ lying in
$B(G_m)$ and defined by the right-hand side of (\ref{Pim}). Set
\begin{equation}\label{equation48}
\mathbf{\Pi}:=\underset{\to}\lim\Pi_m,
\end{equation}
where $\Pi_m:=\{\mathbb{P}^{k_m-c_m-1}\subset B_m\ |\ \mathbb{P}^{k_m-c_m-1}$ is a linear subspace of some
$\mathbb{P}^{k_m-1}\in\Pi(G_m)\}$.

Fix $m$ and let $\pi:=p_m|_{q_m^{-1}(X_m)}:q_m^{-1}(X_m)\to \tilde{G}_m$
be the projection. Consider the relative grassmannian
$G_\pi:=\{\mathbb{P}^1\subset q_m^{-1}(X_m)\ |\ \mathbb{P}^1$
lies linearly in a fibre of the projection $\pi\}$ with induced projections
$\rho:G_\pi\to \tilde{G}_m$ and $q_\pi:G_\pi\to B_m$.
By definition, the fibre $\rho^{-1}(x)$ over an arbitrary point $x\in \tilde{G}_m$ is the grassmannian of projective lines in the projective space $\pi^{-1}(x)$.
(Note that for a point $x\in \tilde{G}_m$ in general position the fibre $\pi^{-1}(x)$
is a projective space
$\mathbb{P}^{k_m-c_m}$, hence $\rho^{-1}(x)\simeq G(2,\mathbb{C}^{k_m-c_m+1})$.)
Furthermore, the projection $q_{\pi}$ is birational.

By construction, the projective spaces $\mathbb{P}^{k_m-c_m-1}\in\Pi_m$ are isomorphic
images under $q_\pi$ of projective spaces $\mathbb{P}^{k_m-c_m-1}$
lying linearly in the fibres of $\rho:G_\pi\to\tilde{G}_m$. Considering the set
$G_{\rho}:=\{\mathbb{P}^{k_m-c_m}\subset q_m^{-1}(X_m)\ |\ \mathbb{P}^{k_m-c_m}$ lies linearly in a fibre of
$\pi:q_m^{-1}(X_m)\to \tilde{G}_m\}$, we obtain a
$\mathbb{P}^{k_m-c_m}$-fibration $q_{\rho}:\Pi_m\to G_{\rho}$ with fibre
$q_{\rho}^{-1}(y)=\mathbb{P}^{k_m-c_m}$ over a given point $y=\{\mathbb{P}^{k_m-c_m}\}\in G_{\rho}$.
Trerefore, to check the irreducibility of $\Pi_m$ it suffices to check  the irreducibility of $G_{\rho}$.
Note that the projection $\pi$ induces a projection $\tau:G_\rho\to \tilde{G}_m$ such that the fibre of $\tau$ over a point $V_{k_m+1}\in \tilde{G}_m$ coincides with the
grassmannian $G(k_m-c_m+1,W)$, where $W\subset V_{k_m+1}$ is the subspace
defined by the condition $\mathbb{P}(W)=\pi^{-1}(x)$. As above,
(\ref{equation42}) implies that, for $i\ge1$, the locally closed subsets
$\tau^{-1}(Z_i(W))$ of $G_{\rho}$ have dimensions strictly less than that of the open
subset $\tau^{-1}(Z(W)\setminus Z_i(W))$. This proves the irreducibility of $G_{\rho}$,
hence of $\Pi_m$.

Next, (\ref{Bm(x)}) implies that, for any point $x=V_{k_m}\in X_m$ and $\tilde{m}\ge m$,
the base $B_m(x)$ of the family of projective lines on $X_m$
passing through $x$ is a linear section of the variety
$\mathbb{P}((\phi_{{\tilde m}-1}\circ...\circ\phi_m)(V_{k_m})^*)\times
GO(1,(\phi_{{\tilde m}-1}\circ...\circ\phi_m)(V_{k_m})^\perp/(\phi_{{\tilde m}-1}\circ...\circ\phi_m)(V_{k_m}))
\subset\mathbb{P}(V^*_{k_{\tilde{m}}}\otimes V^\perp_{k_{\tilde{m}}}/V_{k_{\tilde{m}}})$
by a projective subspace of codimension $c_{{\tilde{m}}}$ in
$\mathbb{P}(V^*_{k_{\tilde{m}}}\otimes V^\perp_{k_{\tilde{m}}}/V_{k_{\tilde{m}}})$.
Let $b_{{\tilde{m}}}(x):B_{{\tilde{m}}}(x)\to Q_{(\tilde{m})}(x):=
GO(1,V^\perp_{k_{\tilde{m}}}/V_{k_{\tilde{m}}})$
be the natural projection. Note that the fibres of $b_{\tilde{m}}(x)$ are
projective spaces of dimension at least $k_{\tilde{m}}-c_{\tilde{m}}-1$,
and, for points $x$ of $X_m$ and
$z\in Q_{\tilde{m}}(x)$ in general position the fibre $b_{\tilde{m}}(x)^{-1}(z)$
is a projective space
$\mathbb{P}^{k_{\tilde{m}}-c_{\tilde{m}}-1}$
by the Bertini Theorem. Moreover, we have an ind-variety
$\mathbf{B}(x)=\underset{\to}\lim B_{\tilde{m}}(x)$, ${\tilde{m}}\ge m$.

In a similar way we obtain that
$\mathbf{\Pi}(x):=\{\mathbf{P}^\infty\in \mathbf{\Pi}\ |\ \mathbf{P}^\infty\ni x\}$
is the ind-variety $\underset{\to}\lim \Pi_{\tilde{m}}(x),$ ${\tilde{m}}\ge m$,
where
$\Pi_{\tilde{m}}(x):=\{\mathbb{P}^{k_{\tilde{m}}-c_{\tilde{m}}-1}\subset B_m|$
$\mathbb{P}^{k_{\tilde{m}}-c_{\tilde{m}}-1}$ lies as a linear projective subspace in a fibre of the
projection $b_{\tilde{m}}(x)\}$. Let $p(x):\Pi_{\tilde{m}}(x)\to Q_{(\tilde{m})}(x)$
be the induced projection. By construction, for any point
$z\in Q_{(\tilde{m})}(x)$ the fibre $p(x)^{-1}(z)$ is the grassmannian
$G(k_{\tilde{m}}-c_{\tilde{m}},\mathbb{C}^{\dim(b_{\tilde{m}}(x)^{-1}(z))+1}).$
(In particular, this grassmannian is just a point for
$x\in X_{\tilde{m}}$ and $z\in Q_{(\tilde{m})}(x)$ in general position.)
This implies the property (A.ii) since $ Q_{(\tilde{m})}(x)$ is an irreducible quadric hypersurface.

The property (A.iii) is evident. As for the property (A.iv), let $Y$
be a fixed variety and $f:\Pi_{\tilde{m}}(x)\to Y$ be a morphism. For
$\tilde{m}\to\infty$ the fibres of $p(x)$ are either points or are grassmannians whose
dimensions tend to infinity.
Therefore $f$ maps each fibre of $p(x)$ to a point, i.e. $f$ factors through the induced
morphism $g:Q_{(\tilde{m})}(x)\to Y$.
As $Q_{(\tilde{m})}(x)$ is a smooth quadric hypersurface whose dimension
tends to infinity as ${\tilde{m}}\to\infty$, it follows that for large enough ${\tilde{m}}$ the morphism $g$ is a constant map. Hence, $f$ is constant too, and (A.iv) is proved.

Theorem \ref{E on linear X} yields now the following.

\begin{theorem}\label{theorem5.6}
A vector bundle on a linear section $\mathbf{X}$ of small codimension of
$\mathbf{G}(\infty),$ $\mathbf{G}\mathrm{O}(\infty,\infty)$, $\mathbf{G}\mathrm{S}(\infty,\infty)$
is isomorphic to a direct sum of line bundles $\mathcal{O}_{\mathbf{X}}(a_i)$ for
$a_i\in\mathbb{Z}$.
\end{theorem}

\end{sub}

\vspace{1cm}
\section{Ind-products and their subvarieties satisfying the properties L, A, T}\label{section6}

\vspace{5mm}

\begin{sub}{\bf Finite or countable ind-products satisfying the properties L, A, T.}\label{subsection 6.1}

\rm
Let $\mathbf{X}^{\xi}=\underset{\to}\lim X_m^{\xi}$, ${\xi}\in\Xi$, be a
countable collection of ind-varieties.
We assume that for each ${\xi}\in\Xi$ and each $m\ge1$ we have a fixed inclusion
$X_m^{\xi}\subset X_{m+1}^{\xi}$.
On every $\mathbf{X}^{\xi}$ we fix a point $x_0^{\xi}$.
Without loss of generality we assume that
$x_0^{\xi}\in X_1^{\xi}$.
Fix a bijection
$\nu:\mathbb{N}\overset{\sim}\to\Xi$ and denote $\underline{m}:=\{1,2,...,m\}$.
Set
$$
{}_{\nu}X_m:=\underset{\xi\in\nu(\underline{m})}\times X_m^{\xi}
$$
and consider the embeddings
$$
{}_{\nu}X_m\hookrightarrow{}_{\nu}X_{m+1}={}_{\nu}X_m\times X_m^{\nu(m+1)},\ \
x\mapsto\big(x,x_0^{\nu(m+1)}\big),\ m\ge1.
$$
We call the ind-variety $\mathbf{X}:=\underset{\to}\lim{}_{\nu}X_m$
an {\it ind-product} of the ind-varieties
$\{\mathbf{X}^{\xi}\}_{\xi\in\Xi}$
and denote it as
\begin{equation}\label{def prod}
\mathbf{X}=\underset{\xi\in\Xi}\times\mathbf{X}^{\xi}.
\end{equation}

Note that $\mathbf{X}$ does not depend, up to an
isomorphism of ind-varieties, on the choice of the bijection
$\nu:\mathbb{N}\overset{\sim}\to\Xi$,
and thus the notation (\ref{def prod}) is consistent. Indeed, let
$\nu':\mathbb{N}\overset{\sim}\to\Xi$
be another bijection, and let
$\psi:=\nu'^{-1}\circ\nu:\mathbb{N}\overset{\sim}\to\mathbb{N}$
be the induced bijection. An isomorphism
$\mathbf{f}:\lim{}_{\nu}X_m\overset{\sim}\to\lim{}_{\nu'}X_m$,
$\mathbf{f}=\{f_m:{}_{\nu}X_m\to{}_{\nu'}X_{\tilde{m}(m)}\}$,
and its inverse
$\mathbf{g}=\mathbf{f}^{-1}:\lim{}_{\nu'}X_m\overset{\sim}\to\lim{}_{\nu}X_m$,
$\mathbf{g}=\{g_m:{}_{\nu'}X_m\to{}_{\nu}X_{\tilde{m}(m)}\}$,
for
$\tilde{m}(m):=\underset{m'>m}\min\{m'\ |\ \psi(\underline{m})\subset\underline{m}'\}$,
are given by the formulas
$$
f_m:{}_{\nu}X_m=\underset{\xi\in\nu(\underline{m})}\times X_m^{\xi}\to{}_{\nu'}X_{\tilde{m}(m)}=
(\underset{\xi\in\nu(\underline{m})}\times X_{\tilde{m}(m)}^{\xi})\times
(\underset{\xi\in\nu'(\underline{\tilde{m}(m)})\setminus\nu(\underline{m})}\times X_{\tilde{m}(m)}^{\xi}),
$$
$$
x\mapsto\big(x,\{x_0^{\xi}\}_{\xi\in\nu'(\underline{\tilde{m}(m)})
\setminus\nu(\underline{m})}\big),
$$
$$
g_m:{}_{\nu'}X_m=\underset{\xi\in\nu'(\underline{m})}\times X_m^{\xi}\to{}_{\nu}X_{\tilde{m}(m)}=
(\underset{\xi\in\nu'(\underline{m})}\times X_{\tilde{m}(m)}^{\xi})\times
(\underset{\xi\in\nu(\underline{\tilde{m}(m)})\setminus\nu'(\underline{m})}\times X_{\tilde{m}(m)}^{\xi}),
$$
$$
x\mapsto\big(x,\{x_0^{\xi}\}_{\xi\in\nu(\underline{\tilde{m}(m)})
\setminus\nu'(\underline{m})}\big).
$$
Note in addition that in principle $\mathbf{X}$ depends on the choice of points $x_0^{\xi}$,
however we suppress this dependence in the notation (\ref{def prod}).

The reason we call $\mathbf{X}$ an ind-product rather than a product is that
$\mathbf{X}$ is not a direct product in the category of ind-varieties. Of course, there
are well-defined projections of ind-varieties
$\mathbf{p}_{\xi}: \mathbf{X}\to\mathbf{X}^{\xi}$,
$\mathbf{p}_{\xi}=\underset{\to}\lim p_{\xi m}$,
where
$p_{\xi m}:{}_{\nu}X_m\to X_m^{\xi}$
is the projection onto the ${\xi}$-factor for ${\xi}\in\nu(\underline{m})$, and the
constant map
$p_{\xi m}:{}_{\nu}X_m\to x_0^{\nu(m)}$ for ${\xi}\not\in\nu(\underline{m})$.
However, $\mathbf{X}$ fails to satisfy the universality property of a product.

If $\Xi$ is finite, we define the ind-product $\underset{\xi\in\Xi}\times\mathbf{X}^{\xi}$
as the set-theoretic direct product of the $\mathbf{X}^{\xi}$'s.
Then $\underset{\xi\in\Xi}\times\mathbf{X}^{\xi}=\underset{\to}\lim(\underset{\xi\in\Xi}\times{X}_m^{\xi})$,
and $\underset{\xi\in\Xi}\times\mathbf{X}^{\xi}$ clearly satisfies the universality property of a direct product in the category of ind-varieties.

Now let $\Xi$ be finite or countable and let the ind-varieties
$\mathbf{X}^{\xi}=\underset{\to}\lim X_m^{\xi}$
satisfy the properties L, A and T. This means that on each $\mathbf{X}^{\xi}$
there exists a collection
$\mathbf{L}^{\xi}:=\{\mathbf{L}_{\theta}^{\xi}\ |\ \theta\in\Theta_{\mathbf{X}^{\xi}}\}$
of line bundles and a collection
$\mathbf{B}^{\xi}:=\{\mathbf{B}_{\theta}^{\xi}=
\underset{\to}\lim B_{\theta m}^{\xi}\ |\ \theta\in\Theta_{\mathbf{X}^{\xi}}\}$
of ind-varieties of projective lines such that $\mathbf{X}^{\xi}$ satisfies
the properties L, A and T. The collections $\{\mathbf{L}^{\xi}\}_{\xi\in\Xi}$ yield the following
countable collection of line bundles on $\mathbf{X}$
\begin{equation}\label{equation55}
\mathbf{L}=\{\mathbf{p}_{\xi}^*\mathbf{L}_{\theta}^{\xi}\ |\
\theta\in\Theta_{\mathbf{X}^{\xi}},\ \xi\in\Xi\}.
\end{equation}
Moreover, any projective line $\mathbb{P}^1$ on $\mathbf{X}^{\xi}$ determines a projective
line $\mathbb{P}^1_{\mathbf{X}}$ on $\mathbf{X}$ such that
$\mathbf{p}_{\xi}(\mathbb{P}^1_{\mathbf{X}})=\mathbb{P}^1$
and
$\mathbf{p}_{\xi'}(\mathbb{P}^1_{\mathbf{X}})=\{x_0^{\xi'}\}$ for $\xi'\ne\xi$.
Therefore each ind-variety $\mathbf{B}_{\theta}^{\xi}$
of projective lines on $\mathbf{X}^{\xi}$ "lifts" to an ind-variety of projective lines
on $\mathbf{X}$. In this way we obtain a collection $\mathbf{B}$ of ind-varieties of
projective lines on $\mathbf{X}$.
Since each $\mathbf{X}^{\xi}$
satisfies the properties L, A and T, it is easy to check that
$\mathbf{X}$ satisfies the same properties with respect to the collections
$\mathbf{L}$ and $\mathbf{B}$.

This, together with Theorem \ref{E on linear X}, leads to the following theorem.

\begin{theorem}\label{theorem6.1}
A vector bundle on $\mathbf{X}=\underset{\xi\in\Xi}\times \mathbf{X}^{\xi}$,
where each ind-variety $\mathbf{X}^{\xi}$ satisfies the properties L, A and T,
is isomorphic to a direct sum of line bundles.
\end{theorem}

\end{sub}

\begin{sub}{\bf Linear sections of ind-products.}\label{multilin secns}

\rm
In this subsection we assume that $\Xi$ is finite, $\Xi=\{1,2,...,l\}$, and that the
ind-varieties
$\mathbf{X}^i=\underset{\to}\lim X_m^i$, $i\in\Xi$,
are copies of the standard ind-grassmannians
$\mathbf{G}(\infty),$ $\mathbf{G}\mathrm{O}(\infty,\infty)$,
$\mathbf{G}\mathrm{S}(\infty,\infty)$. By the above,
$\underset{i\in\Xi}\times\mathbf{X}^i$
is a direct limit $\underset{\to}\lim(\underset{i\in\Xi}\times X_m^i)$.
Each ${X}_m^i$ is a (possibly isotropic) grassmannian which we consider as lying via the Pl\"ucker embedding in
$\mathbb{P}^{N_{im}-1}=\mathbb{P}(V_{N_{im}})$,
and the embeddings
$\underset{i\in\Xi}\times X_m^i\hookrightarrow\underset{i\in\Xi}\times X_{m+1}^i$
are induced by standard extensions
${X}_m^i \hookrightarrow {X}_{m+1}^i$.
We also assume that
$\underset{i\in\Xi}\times X_m^i$ lies in $\mathbb{P}^{N_{m}-1}$ via the Segre embedding
$\underset{i\in\Xi}\times \mathbb{P}^{N_{im}-1}\hookrightarrow \mathbb{P}^{N_{m}-1}.$

For each $m\ge1$ and $i\in\Xi$ set
$$
\widehat{X}_m^i:=X_m^1\times... \times X_m^{i-1}\times X_m^{i+1}\times ...\times X_m^l,
$$
and for $X_m^i=G(k_{im},V_{n_{im}})$ (respectively, $X_m^i=GO(k_{im},V_{n_{im}})$ or
$X_m^i=GS(k_{im},V_{n_{im}})$), set ${X_m^i}^+:=G(k_{im}+1,V_{n_{im}})$ (respectively,
${X_m^i}^+=GO(k_{im}+1,V_{n_{im}})$ or ${X_m^i}^+=GS(k_{im}+1,V_{n_{im}})$).
Consider the flag variety
$\Sigma_{im}:=\{(V_{k_{im}}, V_{k_{im}+1})\in X_m^i\times {X_m^i}^+\ |\ V_{k_{im}}\subset V_{k_{im}+1}\}$
with natural projections
${X_m^i}^+{\leftarrow}\Sigma_{im}\to X_m^i$.
There are induced projections
\begin{equation}\label{equation58}
{X_m^i}^+\times\widehat{X}_m^i\overset{p_{im}}{\leftarrow}
\Sigma_{im}\times\widehat{X}_m^i\overset{q_{im}}\to X_m^i\times \widehat{X}_m^i\simeq \widehat{X}_{m},\ \
i\in\Xi,
\end{equation}
and we put
\begin{equation}\label{equation59}
\pi_{im}:= p_{im}|_{q^{-1}_{im}(X_m)}:q^{-1}_{im}(X_m)\to
{X_m^i}^+\times\widehat{X}_m^i,\ \ i\in\Xi.
\end{equation}

Now let $\{c_m\}_{m\ge1}$ be a nondecreasing sequence of integers satisfying the conditions
\begin{equation}\label{equation56}
1\le c_m\le {\rm min}\{\frac{k_{im}-1}{2}\ |\ i\in\Xi\},
\end{equation}
where
$X_m^i=G(k_{im},V_{n_{im}}),$ $GO(k_{im},V_{n_{im}}),$  $GS(k_{im},V_{n_{im}}).$
For each $m\ge1$
consider a linear section $X_m$ of $\underset{i\in\Xi}\times X_m^i$,
$$
X_m:=(\underset{i\in\Xi}\times X_m^i)\cap \mathbb{P}^{N_{m}-c_m-1},
$$
where
$\mathbb{P}^{N_{m}-c_m-1}=\mathbb{P}(U_m)$ for $U_m\in G(N_{m}-c_m,V_{N_{m}})$. We call $X_m$ a {\it linear section of} $\underset{i\in\Xi}\times X_m^i$ {\it of codimension} $c_m.$

\begin{proposition}\label{proposition6.2}
For a given $m\ge 1$ such that $k_{im}\ge2$ for all $i\in\Xi,$
fix an integer $c_m$ satisfying {\rm (\ref{equation56})}.
Then for a projective subspace $\mathbb{P}^{N_{m}-c_m-1}=\mathbb{P}(U_m)$
of general position in
$\mathbb{P}^{N_{m}-1},$ $U_m\in G(N_{m}-c_m,V_{N_{m}})$, and any $i\in\Xi$
the following statements hold:

(i) the varieties $X_m$ and $q^{-1}_{im}(X_m)$ are smooth;

(ii) ${\rm codim}_{\underset{i\in\Xi}\times X_m^i} X_m=
{\rm codim}_{\Sigma_{im}\times\widehat{X}_m^i}q^{-1}_{im}(X_m)=c_m$,
$\pi_{im*} \mathcal{O}_{q^{-1}_{im}(X_m)}=\mathcal{O}_{{X_m^i}^+\times\widehat{X}_m^i}$,
and, for a point
$x=(x_1,x_2)\in{X_m^i}^+\times\widehat{X}_m^i$ in general position,
the projective subspace
$\mathbb{P}^{k_{im}}=q_{im} p_{im}^{-1}(x)$ of $X_m^i\times\{x_2\}$
satisfies the condition
$$
{\rm codim}_{\mathbb{P}^{k_{im}}}(\mathbb{P}^{k_{im}}\cap \mathbb{P}^{N_{m}-c_m-1})=c_m,
$$
so that $Z_m^i(U_m):=\{x\in{X_m^i}^+\times\widehat{X}_m^i\ |\ \dim\pi_{im}^{-1}(x)> k_{im}-c_m\}$ is a proper closed
subset of ${X_m^i}^+\times\widehat{X}_m^i$;

(iii) let $B_m^i(U_m):=\pi_{im}^{-1}(Z_m^i(U_m));$
then
${\rm codim}_{q^{-1}_{im}(X_m)} B_m^i(U_m)\ge 2$,
${\rm codim}_{{X_m^i}^+\times\widehat{X}_m^i}Z_m^i(U_m)\ge 3$;

(iv) the projection
$\pi_{im}:q^{-1}_{im}(X_m)\setminus B_m^i(U_m)\to
{X_m^i}^+\times\widehat{X}_m^i\setminus Z_m^i(U_m)$
is a projective $\mathbb{P}^{k_{im}-c_m}$-bundle.
\end{proposition}

\begin{proof}
Similar to the proof of Proposition \ref{proposition5.1}.
\end{proof}

\begin{corollary}\label{corollary6.3}
Under the assumptions of Proposition {\rm \ref{proposition6.2}},
let $i\in\Xi$ and let $\mathcal{E}$ be a vector bundle on $q_{im}^{-1}(X_m)$
trivial along the fibres of the morphism
$\pi_{im}:q_{im}^{-1}(X_m)\to{X_m^i}^+\times\widehat{X}_m^i$.
Then the sheaf $\pi_{im*}\mathcal{E}$ is locally free and
$ev:\pi_{im}^*\pi_{im*}\mathcal{E}\overset{\simeq}\to \mathcal{E}$
is an isomorphism.
\end{corollary}

\begin{proof}
Proposition \ref{proposition6.2} implies that the data
$X={X_m^i}^+\times\widehat{X}_m^i$, $Y=q_{im}^{-1}(X_m)$,
$B=B_m^i(U_m)$, $Z=Z_m^i(U_m)$, $E=\mathcal{E}$
satisfy the conditions of Proposition 7.2 from the appendix.
Therefore this latter proposition yields the corollary.
\end{proof}

Below we will need the following lemma, the proof of which is similar to that of Lemma~\ref{lemma5.3}.

\begin{lemma}\label{lemma6.4}
For $i=1,2,$ let $X_i=G(k_i,V_i),$ $GO(k_i,V_i),$ $GS(k_i,V_i)$ and let $X$ be a linear section of
codimension $c$ of $X_1\times X_2$, where $1\le c\le {\rm min}\{\frac{k_1+1}{2},\frac{k_2+1}{2}\}$. Then
for any projective line $\mathbb{P}^1_i\subset X_i$ there exists a projective line $\mathbb{P}^1$ such that $pr_i|_{\mathbb{P}^1}$ is an isomorphism of $\mathbb{P}^1$ and $\mathbb{P}^1_i$. Here $pr_i$ stands for the natural projection $X_1\times X_2\to X_i.$
\end{lemma}

Consider an ind-variety $\mathbf{X}=\underset{\to}\lim X_m$ such that, for each
$m\ge1$, $X_m$
is a smooth linear of codimension
$c_m$ of $\underset{i\in\Xi}\times X_{im}$ where $c_m$ satisfies (\ref{equation56}),
and the embeddings
$\phi_m:X_m \hookrightarrow X_{m+1}$
are induced by the corresponding embeddings
$\underset{i\in\Xi}\times X_m^i\hookrightarrow\underset{i\in\Xi}\times X_{m+1}^i$.
In what follows we call $\mathbf{X}$ a {\it linear section of } $\underset{i\in\Xi}\times \mathbf{X}^i$
{\it of small codimension.}
The existence of linear sections ${\mathbf{X}}$ of $\underset{i\in\Xi}\times \mathbf{X}^i$ of small codimension follows immediately from the Bertini Theorem.

\begin{theorem}\label{theorem6.5}
Let $\mathbf{X}$ be a linear section of $\underset{i\in\Xi}\times \mathbf{X}^i$ of small codimension. Then a vector bundle on $\mathbf{X}$ is a direct sum of line bundles.
\end{theorem}

\begin{proof}
We give a proof for the case when all $\mathbf{X}^i$ are isomorphic to
$\mathbf{G}\mathrm{O}(\infty,\infty)$.
The case when some $\mathbf{X}^i$ are isomorphic to $\mathbf{G}(\infty)$ or $\mathbf{G}\mathrm{S}(\infty,\infty)$
is treated similarly.

First we construct families $\mathbf{B}_i$, $i\in\Xi$, on $\mathbf{X}$. For each $m\ge1$ and each $i$,
$1\le i\le l$, consider the natural projection
$p_{im}:X_m\to \hat{X}_m^i$,
and for an arbitrary point
$y\in \hat{X}_m^i$ set $X_m^i(y):=p^{-1}_{im}(y)$.
By definition, $X_m^i(y)$ is a linear section of the grassmannian
$GO(k_{im},V_{n_{im}})$. Let $B_i(m)(y)$ be the family of projective lines in
$GO(k_{im},V_{n_{im}})$ lying on $X_m^i(y)$,
and set
$B_i(m):=\underset{y\in \hat{X}_m^i}\cup B_i(m)(y)$.
Then
$\mathbf{B}_i:=\underset{\to}\lim B_i(m)$, $i\in\Xi$.
Furthermore, the ind-varieties
$\mathbf{\Pi}_i$, $i\in\Xi,$
parametrizing certain families of ind-projective spaces
$\mathbb{P}^\infty$ in $\mathbf{B}_i$, are
defined in the same way as the ind-variety $\mathbf{\Pi}$
in subsection 5.2 -- see (\ref{equation48}).

Next, recall the collection of line bundles (\ref{equation55}) on $\underset{i\in\Xi}\times\mathbf{X}^i$.
In our case $\Theta_{\mathbf{X}^{\xi}}$ consists of a single point for each $i$, hence we can write simply
$\mathbf{L}=\{\mathbf{L}_i\}_{i\in\Xi}$. We now define a family of line bundles
$\mathbf{L}_{\mathbf{X}}$ by putting
$\mathbf{L}_{\mathbf{X}}:=\{\mathbf{L}_i|_{\mathbf{X}}\}_{i\in\Xi}$.
Then by the Lefschetz Theorem $\mathbf{L}_{\mathbf{X}}$ freely generates $\mathrm{Pic}\mathbf{X}$;
in addition, the relation (\ref{Li|Pj}) is clearly satisfied.
To see that $\mathbf{X}$ satisfies the property L, it remains to notice that
$H^1(X_m,\underset{i\in\Xi}\otimes\mathbf{L}_i^{\otimes a_i}|_{X_m})=0$ for all $a_i$.
Indeed, the vanishing of
$H^1(\underset{i\in\Xi}\times X_m^i,
\underset{i\in\Xi}\otimes\mathbf{L}_i^{\otimes a_i}|_{\underset{i\in\Xi}\times X_m^i})$
follows from Kunneth's and Bott's formulas. Since
$\underset{i\in\Xi}\otimes\mathbf{L}_i^{\otimes a_i}|_{X_m}$ admits a Koszul resolution similar to (\ref{equation32}),
this is sufficient to conclude that
$H^1(X_m,\underset{i\in\Xi}\otimes\mathbf{L}_i^{\otimes a_i}|_{X_m})=0$.

Using Proposition \ref{proposition6.2} and
repeating the argument from subsection 5.2, it is easy to check that $\mathbf{X}$ satisfies the property A.
Let us show that $\mathbf{X}$ satisfies the property T.
The case $|\Xi|=1$ was treated in Proposition 4.4(ii). It is enough to give the proof for the case
$|\Xi|=2$;  the proof for $|\Xi|\ge 3$ goes along the same lines. We thus assume that $X_m$ is a linear section
of $X_m^1 \times X_m^2$. According to (\ref{equation58}) and (\ref{equation59}) we have a commutative diagram
\begin{equation}\label{diag1}
\xymatrix{
q^{-1}_{im}(X_m)\ar[d]_{\pi_{1m}}\ar[r]^-{q_{1m}}&
X_m\ar[d]^{\rho}\\
{X_m^1}^+\times X_m^2\ar[r]_-{\lambda} & X_m^2, }
\end{equation}
where $l=2$, $i=1$, so that $\hat{X}_m^1=X_m^2$, and $\lambda$ and $\rho$
are the natural projections.

Let $\underset{\leftarrow}\lim E_m$ be a $\mathbf{B}_i$-trivial vector bundle on
$\mathbf{X}=\underset{\to}\lim X_m$ for $i=1,2$.
This means that each vector bundle $E_m$ is a $B_i(m)$-trivial bundle on $X_m$, i.e.
$E_m|_{\mathbb{P}^1}$ is trivial for any $\mathbb{P}^1\in B_i(m)$, $i=1,2$. Consider the vector bundle
$$
\tilde{E}_m:=q^*_{1m}E_m.
$$
Since $E_m$ is $B_1(m)$-trivial, i.e. linearly trivial on the fibres of $\rho$, $E_m$ is trivial by Proposition
\ref{proposition5.4}. It follows that $\tilde{E}_m$ is trivial along the fibres of $\pi_{1m}$.
Therefore, by Corollary \ref{corollary6.3} there is an isomorphism
$$
ev:\pi^*_{1m}\pi_{1m*}\tilde{E}_m\overset{\simeq}\to \tilde{E}_m.
$$
Moreover, as in the proof of Proposition \ref{proposition5.4}, we obtain that the bundle
$\pi_{1m*}\tilde{E}_m|_{\lambda^{-1}(y)}$ is trivial for any $y\in X_m^2$.

Thus Proposition 7.1 below and diagram (\ref{diag1}) imply that
$E_m^{\lambda}:=\lambda_*\pi_{1m*}\tilde{E}_m$
is a vector bundle on $X_m^2$ such that
$\tilde{E}_m\simeq\pi^*_{1m}\lambda^*E_m^{\lambda}\simeq q^*_{1m}\rho^*E_m^{\lambda}$.
Applying again Proposition 7.1 we obtain
\begin{equation}\label{equation64}
E_m\simeq q_{1m*}\tilde{E}_m\simeq q_{1m*}q^*_{1m}\rho^*E_m^{\lambda}\simeq \rho^*E_m^{\lambda}.
\end{equation}
Since $E_m$ is $B_2(m)$-trivial, it follows from (\ref{equation64}) and Lemma \ref{lemma6.4} that
$E_m^{\lambda}$ is a linearly trivial bundle on $X_m^2$. Hence, $E_m^{\lambda}$
is trivial by Proposition \ref{proposition7.4} below, and $E_m$ is trivial as well. In this we showed that
$\mathbf{X}$ satisfies the property T.

\end{proof}

\end{sub}

\vspace{1cm}

\begin{sub}{\bf Ind-varieties of generalized flags.}\label{gen flags}
\vspace{5mm}

\rm

We first recall some basic definitions concerning generalized flags in a vector space, see
\cite[sections 3-5]{DiP}. Let $V$ be a countable-dimensional vector space.
A set $\mathcal{C}$ of pairwise distinct subspaces of $V$ is called a \textit{chain} if it
is linearly ordered by inclusion. A chain $\mathcal{F}$ of subspaces of $V$ is a
\textit{generalized flag} in $V$ if it satisfies the following conditions:

(i) each $F\in\mathcal{F}$ has an immediate successor or an immediate predecessor, i.e.
$\mathcal{F}=\mathcal{F}'\cup\mathcal{F}''$, where $\mathcal{F}'\subset\mathcal{F}$
(respectively, $\mathcal{F}''\subset\mathcal{F}$) is the set of subspaces in
$\mathcal{F}$ having an immediate successor (respectively, predecessor);

(ii) $V\setminus\{0\}=\sqcup_{F'\in\mathcal{F}'}(F''\setminus F')$,
where $F''\in\mathcal{F}''$ is the immediate successor of $F'\in\mathcal{F}'$.

We define a \textit{flag} in $V$ to be a generalized flag in $V$
which is isomorphic as an ordered set to a subset of $\mathbb{Z}$.
A flag can be equivalently defined as a chain of subspaces of $V$ such that
$\cap_{F\in\mathcal{F}}F=0$, $\cup_{F\in\mathcal{F}}F=V$
and there exists a strictly monotonic map of ordered sets
$\phi:\mathcal{F}\to\mathbb{Z}$.

If $\mathcal{F}$ is a generalized flag in $V$ and $\{e_{\alpha}\}_{\alpha\in A}$ is a
basis of $V$ ($A$ being a countable set), we say that $\mathcal{F}$ and
$\{e_{\alpha}\}_{\alpha\in A}$
are \textit{compatible} if there exists a strict partial order $\prec$ on $A$ such that, for any $F'\in\mathcal{F}'$,
$F'=\mathrm{Span}\{e_{\beta}\ |\ \beta\prec\alpha\}$ for a certain $\alpha\in A$,
and
$F''=F'\oplus\mathrm{Span}\{e_{\gamma}\ |\ \gamma$
is not $\prec$-comparable to $\alpha\}$.

For the rest of this section we fix a basis $E=\{e_n\}$ of $V$. We call a generalized flag
$\mathcal{F}$ \textit{weakly compatible with} $E$ if $\mathcal{F}$
is compatible with a basis $L$ of $V$ such that $E\setminus(E\cap L)$ is a finite set. Furthermore, we define two generalized flags $\mathcal{F}$ and
$\mathcal{G}$ in $V$ to be $E$-\textit{commensurable} if both $\mathcal{F}$ and
$\mathcal{G}$ are weakly compatible with $E$ and there exists an inclusion preserving bijection $\phi:\mathcal{F}\to\mathcal{G}$ and a finite-dimensional subspace $U$ in $V$ such that, for every $F\in\mathcal{F}$,

(i) $F\subset\phi(F)+U,\ \phi(F)\subset F+U$,\ \  and

(ii) $\dim(F\cap U)=\dim(\phi(F)\cap U)$.

Let $\mathcal{F}l(\mathcal{F},E)$ be the set of all generalized flags in $V$ that are $E$-commensurable with $\mathcal{F}$. Following \cite[Prop. 5.2]{DiP}) we endow
$\mathcal{F}l(\mathcal{F},E)$ with a structure of an ind-variety in the following way. For any
$m\ge 1$ denote $E_m:=\{e_\alpha\}_{\alpha\le m}$,
$V_m:=\Span(E_m)$, $E_m^c:=\{e_\alpha\}_{\alpha>m}$, $V_m^c:=\Span(E_m^c)$.
Next, for any $\mathcal{G}\in \mathcal{F}l(\mathcal{F},E)$ choose a positive integer $m_{\mathcal{G}}$
such that
$\mathcal{F}$ and $\mathcal{G}$
are compatible with bases containing $E^c_{m_{\mathcal{G}}}$, and
$V_{m_{\mathcal{G}}}$ contains a finite-dimensional subspace
$U$ which together with the corresponding inclusion preserving bijection
$\phi_{\mathcal{G}}:\mathcal{F}\overset{\sim}\to \mathcal{G}$ makes
$\mathcal{F}$ and $\mathcal{G}$ $E$-commensurable. We can pick $n_{\mathcal{F}}$ so that
$n_{\mathcal{F}}\le m_{\mathcal{G}}$ for every $\mathcal{G}\in\mathcal{F}l(\mathcal{F},E)$. Set
$$
\mathcal{G}_m:=\{G\cap V_m\ |\ G\in \mathcal{G}\}\quad n\ge m_{\mathcal{G}}.
$$
The type of the flag $\mathcal{F}_n$ yields a sequence of integers
$$
0<d_1^m<...<d_{s_{m-1}}^m< d_{s_m}^m=m,
$$
and let $\mathcal{F}l(d_m,V_m)$ be the usual flag variety of type $d_m=(d_{1m},...,d_{m,s_{m-1}})$ in $V_m$.
Notice that
$s_{m+1}=s_m$ or $s_{m+1}=s_m+1$. Furthermore, in both cases an integer $j_m$
is determined as follows: in the former case
$d_{m+1,i}=d_{m,i}$ for $0\le i<j_m$ and $d_{m+1,i}=d_{m,i}+1$ for $j_m\le i<s_m$,
and in the latter case $d_{m+1,i}=d_{m,i}$ for $0\le i<j_m$ and $d_{m+1,i}=d_{m,i-1}+1$ for $j_m\le i<s_m$.

Now we define a map $\iota_m:\mathcal{F}l(d_m,V_m)\to \mathcal{F}l(d_{m+1},V_{m+1})$
for every $m\ge m_{\mathcal{F}}$. Given a flag
${\mathcal{G}}_m=\{0=G_0^m\subset G_1^m\subset... \subset G_{s_m}^m=V_m\}\in \mathcal{F}l(d_m,V_m)$,
put
$\iota_m(\mathcal{G}_m)=\mathcal{G}_{m+1}:=\{0=G_0^{m+1}\subset G_1^{m+1}\subset...\subset G_{s_m+1}^{m+1}
V_{m+1}\}$,
where
$$
G_i^{m+1}=\begin{cases}
G_i^m &{\rm if}\ \ 0\le i<j_m,\cr
G_i^m\oplus ke_{m+1} &{\rm if}\ j_m\le i\le s_{m+1}\ \ {\rm and}\ \ s_{m+1}=s_m,\cr
G_{i-1}^m\oplus ke_{m+1} &{\rm if}\ j_m\le i\le s_{m+1}\ \ {\rm and}\ \ s_{m+1}=s_m+1.\cr
\end{cases}
$$
The maps $\iota_m$ are closed embeddings of algebraic varieties, and hence
$\underset{\to}\lim\mathcal{F}l(d_m,V_m)$ is an ind-variety.
A bijection between $\mathcal{F}l(\mathcal{F},E)$ and $\underset{\to}\lim\mathcal{F}l(d_m,V_mmm)$ is given by
$$
\tau: \mathcal{F}l(\mathcal{F},E)\overset{\sim}\to \underset{\to}\lim\mathcal{F}l(d_m,V_m),\quad \mathcal{G}\mapsto
\underset{\to}\lim\mathcal{G}_m
$$
-- see [DiP, Prop. 5.2].

Assume now that $\mathcal{F}$ is a flag of subspaces in $V$. Then
$\mathcal{F}=\{...\subset F_i\subset F_{i+1}\subset...\}_{i\in\Theta}$,
where $\Theta$ is one of the four linearly ordered sets
$\{1,...,n\}$, $\mathbb{Z}$, $\mathbb{Z}_{>0}$, $\mathbb{Z}_{<0}$.
Assume in addition that
\begin{equation}\label{assumptn inf}
\dim(F''/F')=\infty
\end{equation}
for all $i\in\Theta$ for which
$i+1\in\Theta$. Denote by $\mathcal{\hat{F}}(i)$ the flag
$\mathcal{\hat{F}}\setminus\{F_i\}=\{...\subset F_{i-1}\subset F_{i+1}\subset...\}$.
There is natural projection
$\pi_i:\mathcal{F}l(\mathcal{F},E)\to\mathcal{F}l(\mathcal{\hat{F}}(i),E)$.
Let $\mathcal{\hat{G}}=\{...\subset G_{i-1}\subset G_{i+1}\subset...\}\in\mathcal{F}l(\mathcal{\hat{F}}(i),E)$
and $\mathcal{G}=\{...\subset G_{i-1}\subset G_i\subset G_{i+1}\subset...\}\in\pi_i^{-1}(\mathcal{\hat{G}})$.
Then the fibre $\pi_i^{-1}(\mathcal{\hat{G}})$ equals
$\mathcal{F}l(G_i/G_{i-1},E(i))$ where $E(i)=(E\cap G_{i+1})\setminus(E\cap G_{i-1})$.
Note that the ind-variety $\mathcal{F}l(G_i/G_{i-1},E(i))$ is isomorphic to the ind-grassmannian
$\mathbf{G}({\infty})$.

Moreover, there is a well-defined line bundle
$\mathbf{L}_i:=\mathcal{O}_{\pi_i}(1):=\underset{\leftarrow}\lim\mathcal{O}_{\pi_{im}}(1)$
on $\mathcal{F}l(\mathcal{F},E)$,
where
$\pi_{im}:\mathcal{F}l(d_m,V_m)\to\mathcal{F}l(\hat{d}_m(i),V_m)$
is the natural projection and $\hat{d}_m(i)$ is defined in the same way as $d_m$ using the flag
$\mathcal{\hat{F}}(i)$ instead of $\mathcal{F}$. The fact that the line bundles $\mathcal{O}_{\pi_{im}}(1)$
yield a well-defined bundle $\mathcal{O}_{\pi_i}(1)$ is established by a straightforward checking using the
explicit form of the embeddings $\iota_m$.

By $\mathbf{B}_i(\mathcal{G})$ we denote the ind-variety of projective lines on $\mathcal{F}l(\mathcal{F},E)$
passing through a point $\mathcal{G}\in\mathcal{F}l(\mathcal{F},E)$ and lying in the fibre of $\pi_i$ which
contains $\mathcal{G}$. Finally, we define the ind-variety $\mathbf{\Pi}_i(\mathcal{G})$ as the ind-variety
$\mathbf{\Pi}(\mathcal{G})$ for the ind-grassmannian $\mathcal{F}l(G_i/G_{i-1},E(i))\simeq\mathbf{G}({\infty})$
as defined in subsection \ref{section 4.3}.

It is easy to check that $\mathcal{F}l(\mathcal{F},E)$ satisfies the properties L, A and T with respect to the data
$\Theta_{\mathcal{F}l(\mathcal{F},E)}:=\Theta$, $\mathbf{L}_i$,
$\mathbf{B}_i:=\underset{\mathcal{G}\in\mathcal{F}l(\mathcal{F},E)}\cup\mathbf{B}_i(\mathcal{G})$,
$\mathbf{\Pi}_i:=\underset{\mathcal{G}\in\mathcal{F}l(\mathcal{F},E)}\cup\mathbf{\Pi}_i(\mathcal{G})$.
As a result, Theorem \ref{E on linear X} implies the following theorem.

\begin{theorem}\label{BVTS for flags}
Let $V$ be a countable-dimensional vector space with basis $E$. Let
$\mathcal{F}$ be a flag in $V$ satisfying (\ref{assumptn inf}) weakly compatible with $E$.
Then any vector bundle on $\mathcal{F}l(\mathcal{F},E)$ is isomorphic to a direct sum of line bundles.
\end{theorem}

It is an interesting question whether the BVTS Theorem holds on any ind-variety of
generalized flags $\mathcal{F}l(\mathcal{F},E)$ under the assumption that the
generalized flag
$\mathcal{F}$ satisfies (\ref{assumptn inf}) for all $F'\in\mathcal{F}'$ and their
respective successors $F''$.

\end{sub}

\vspace{1cm}
\section{Appendix}\label{Appendix}

\vspace{1cm}
In this appendix we collect some general facts about coherent sheaves on projective
varieties and their behaviour under flat projective morphisms, which are used throughout
the paper.

\begin{proposition}\label{trivial on fibres}
Let $p:Y\to X$ be a smooth flat projective morphism of projective
varieties with irreducible fibres.

1) If $E$ is a vector bundle on $Y$, trivial on the fibres of $p$, then the evaluation morphism
$ev:p^*p_* E\to E$ is an isomorphism.

2) If $F$ be a vector bundle on $X$, then the canonical morphism $F\overset{\sim}\to p_*p^*F$ is an isomorphism.
\end{proposition}

\begin{proof}
1) This follows easily from the Base-change Theorem [H, Ch.~III, Cor.~12.9].

2) Consider the Stein factorization
$f:Y\overset{f'}\to X'\overset{g}\to X$ of $f$,
where
$X'=\mathbf{Spec}(f_*\mathcal{O}_{Y})$ and
$f'_*\mathcal{O}_{Y}= \mathcal{O}_{X'}$
(see \cite[Ch. III, Cor.12.9]{H1}). Since $f_*\mathcal{O}_{Y}$
is an invertible sheaf by 1), it follows that $g$ is an isomorphism. Therefore
$f_*\mathcal{O}_{Y}=\mathcal{O}_{X}$.
This, together with the projection formula \cite[Ch. III, Exc. 8.3]{H1}, gives the desired
assertion.
\end{proof}

\begin{proposition}\label{proposition7.2}
Let $\pi:{Y}\to {X}$ be a surjective morphism of smooth irreducible projective varieties such that:

(i) the fibres of $\pi$ are projective spaces;

(ii) the variety ${Z}:=\{x \in {X}\ |\  \dim \pi^{-1}(x)>\dim {Y}-\dim {X}\}$ has \textbf{•}codimension at least $3$ in ${X}$, and the variety
$B:=\pi^{-1}(Z)$ has codimension  at least $2$ in $Y$;

(iii) there exists a vector bundle $F$ on $X\setminus Z$ such that
$\pi:Y\setminus B\simeq \mathbb{P}(F)\to X\setminus Z$
is the structure  map of the projectivized vector bundle $F$.

Next, let $E$ be a vector bundle on $Y$, trivial along the fibres of $\pi.$
Then the $\mathcal{O}_{X}$-sheaf $\pi_*E$ is locally free and the evaluation
morphism $ev:\pi^*\pi_* E\to E$ is an isomorphism.
\end{proposition}

\begin{proof}
We first show that $ev:\pi^*\pi_* E\to E$ is an isomorphism.
For this, consider an arbitrary open subvariety $U\subset X$ and its
closed subvariety $A\subset U$ such that
\begin{equation}\label{equation88}
{\rm codim}_U A\ge2.
\end{equation}
Since $X$ is smooth and $Z$ has codimension $\ge3$ in
$X$, it follows that ${\rm codim}_U(Z\cap A)\ge3$ and
${\rm codim}_{\pi^{-1}(U)}\pi^{-1}(Z\cap A)\ge2$. Next, (iii) and (\ref{equation88}) imply
${\rm codim}_{\pi^{-1}(U)}\pi^{-1}(Z\setminus Z\cap A)\ge2.$ Hence
\begin{equation}\label{equation89}
{\rm codim}_{\pi^{-1}(U)}\pi^{-1}( A)\ge2.
\end{equation}

Let $s\in H^0(U\setminus A,\pi_* E|_{U\setminus A})$ and let $\tilde{s}:=\phi(s)$,
where
$\phi: H^0(U\setminus A,\pi_* E|_{U\setminus A})\overset{\simeq}\to
H^0(\pi^{-1}(U\setminus A),E|_{\pi^{-1}(U\setminus A)})$
is the canonical isomorphism. Since $E$ is a locally free sheaf on a smooth variety
$Y$, $E$ is normal by \cite[Prop. 1.6(ii)]{H2}, i.e. (\ref{equation89}) implies
that $\tilde{s}$ extends uniquely to a section
$\tilde{s}'\in H^0(\pi^{-1}(U),E|_{\pi^{-1}(U)})$.
Then $s$ extends to the section
$s':=\psi(\tilde{s}')\in H^0(U,\pi_* E|_U)$,
where
$\psi:H^0(\pi^{-1}(U),E|_{\pi^{-1}(U)}) \overset{\simeq}\to H^0(U,\pi_* E|_U)$
is the canonical isomorphism. In view of (\ref{equation88}) this means that the sheaf $\pi_* E$ is normal.

Note that $\pi_* E$ is torsion-free. Indeed, if the torsion subsheaf $Tors(\pi_* E)$
were nonzero, then since $E$ is locally free, by (iii) any section
$0\ne s\in H^0(Y,Tors(\pi_* E))$
would be supported in $Z$. Then the section
$0\ne \tilde{s}:=\psi^{-1}(s)$ would be supported in $B$,  i.e. $Tors(E)\ne 0$, This contradicts the assumptions that $E$ is locally free and $Y$ is smooth
and irreducible.

Hence, $\pi_* E$ is reflexive by \cite[Prop. 1.6]{H2}.
Set
$$
\tilde{E}:=\pi^*\pi_*E/Tors(\pi^*\pi_*E).
$$
Proposition \ref{trivial on fibres}, together with (iii), implies the existence of an
isomorphism
\begin{equation}\label{equation91}
\alpha:\tilde{E}|_{Y\setminus B}\overset{\simeq}\to
E|_{Y\setminus B}.
\end{equation}
Now by \cite[Prop. 1.1]{H2} the
sheaf $\pi_* E$ can, locally on $X$, be included in an exact sequence
\begin{equation}\label{equation90}
0\to\pi_*E\to L_1\to L_2,
\end{equation}
with locally free sheaves $L_1$ and $L_2$.
Applying to (\ref{equation90}) the functor $\pi^*$ we obtain the sequence
$$
0\to\tilde{E}\to\pi^*L_1\to\pi^*L_2
$$
which is exact when restricted onto $Y\setminus B$. Hence, this
sequence is exact as $\tilde{E}$ is torsion free and the sheaves $\pi^*L_1$ and
$\pi^*L_2$ are locally free. By \cite[Prop. 1.1]{H2} this implies that $\tilde{E}$ is
reflexive. Therefore, denoting by $i$ the inclusion
$Y\setminus B \hookrightarrow Y$
and using the isomorphism (\ref{equation91}) and \cite[Prop. 1.6(iii)]{H2} we obtain an
isomorphism
$$
\pi^*\pi_*E=\tilde{E}\cong i_*(\tilde{E}|_{Y\setminus
B})\overset{i_*\alpha}{\underset{\simeq}\to} i_*({E}|_{Y
\setminus B})\cong E.
$$
This isomorphism is nothing but the evaluation morphism $ev$.

It remains to show that $\pi_* E$ is locally free. The isomorphism $ev$ implies  that
$\pi^*\pi_*E$ is locally free. Therefore, for any $y\in Y$,
$r:=\dim_{\mathbb{C}(y)}(\pi^*\pi_*E\otimes\mathbb{C}(y))$ does not depend on $y$, and
consequently, $\dim_{\mathbb{C}(x)}(\pi_*E\otimes\mathbb{C}(x))=r$.
According to \cite[\S5, Lemma 1]{M}, since $X$ is smooth, the fact that
$\dim_{\mathbb{C}(x)}(\pi_*E\otimes\mathbb{C}(x))$ does not depend on $x\in X$
implies that $\pi_*E$ is locally free.

\end{proof}

\begin{proposition}\label{trivial on lines}
Let $Q_n$ be a nonsingular $n$-dimensional quadric in $\mathbb{P}^{n+1}$ for $n\ge2$,
and let $E$ be a linearly trivial vector bundle on $Q_n$. Then $E$ is trivial.
\end{proposition}
\begin{proof}
We argue by induction on $n$. For $n=2$ the proof is easy and is
the same as for a projective space as given in \cite[Ch. I, Thm. 3.2.1]{OSS}.
Thus we may assume that $n\ge3$.
Consider a codimension-2 subspace $\mathbb{P}^{n-1}$ in $\mathbb{P}^{n+1}$ such that
$Q_{n-2}:=Q_n\cap \mathbb{P}^{n-1}$ is a smooth quadric of dimension $n-2$.
If $n\ge4$ then $E|_l$ is trivial for any projective line $l\subset Q_{n-2}$, hence
$E|_{Q_{n-2}}$ is trivial by the induction assumption.
For $n=3$ the quadric $Q_{n-2}$ is a smooth conic $C$. By Bertini's Theorem there exists a
smooth quadric surface $Q_2$ on $Q_3$ passing through the conic $C$.
Since $E|_{Q_2}$ is trivial (being linearly trivial), $E|_C$ is also trivial,
i.e. our claim holds for $n=3$.

We will now use the triviality of $E|_{Q_{n-2}}$ for $n\ge4$ to show that $E$ is trivial.
Let $\sigma_Q:\tilde{Q}_n\to Q_n$ be the blow-up of $Q_n$ with center at
$Q_{n-2}$, and let $D:=\sigma_Q^{-1}(Q_{n-2})$ be the exceptional divisor.
Clearly, $D\simeq Q_{n-2}\times\mathbb{P}^1$ and there is a flat surjective morphism
$\pi:\tilde{Q}_n\to\mathbb{P}^1$ fitting in the commutative diagram
\begin{equation}\label{Lefschetz pencil}
\xymatrix{
D=Q_{n-2}\times\mathbb{P}^1\ar[dr]_{pr_2}\ar@{^{(}->}[r]^-{i} & \tilde{Q}_n\ar[d]^-{\pi}\\
& \mathbb{P}^1,}
\end{equation}
where $i$ is the embedding of the exceptional divisor. By construction, there exist two
distinct points
$t_1,t_2\in\mathbb{P}^1$ such that
the fibre $Q_{n-1}(t)=\pi^{-1}(t)$ is a smooth quadric
for $t\in U:=\mathbb{P}^1\smallsetminus\{t_1\cup t_2\}$,
and $Q_{n-1}(t_j):=\pi^{-1}(t_j)$ for $j=1,2$
are quadratic cones whose vertices are points.

Consider the vector bundle $\tilde{E}:=\sigma_Q^*E$ on $\tilde{Q}_n$. By construction,
$\tilde{E}$ is trivial on any projective line $l\subset Q_{n-1}(t),\ t\in U$.
Hence, by the induction assumption,
$E|_{Q_{n-1}(t)},\ t\in U$,
is trivial. Consequently,
\begin{equation}\label{h^i=}
H^i(Q_{n-1}(t),\tilde{E}(-D)|_{Q_{n-1}(t)})=0,\ i\ge0,\ \ \
\dim H^i(Q_{n-1}(t),\tilde{E}|_{Q_{n-1}(t)})=
\left\{\begin{array}{cc}
r, & {\rm if}\ i=0,\\
0, & {\rm if}\ i\ge1,
\end{array}\right.
\ \ \ t\in U.
\end{equation}

Next, for $j=1,2$, let $\sigma:K_j\to Q_{n-1}(t_j)$ be the blow-up of the cone
$Q_{n-1}(t_j)$ with center at the singular point. Let $f_j:K_j\to Q_{n-2}$
be the induced $\mathbb{P}^1$-bundle, the fibres of which map to projective lines on
$Q_{n-1}(t_j)$ under the morphism $\sigma$. Since
$\tilde{E}_{t_j}:=\tilde{E}|_{Q_{n-1}(t_j)}$
is trivial along the projective lines on $Q_{n-1}(t_j)$, it follows that the bundle
$\tilde{E}_{K_j}:=\sigma^*\tilde{E}_{t_j}$
is trivial along the fibers of $f_j$. Therefore, for an arbitrary point
$x\in Q_{n-2}$ we obtain
$$
H^i(Q_{n-2},\tilde{E}_{K_j}\otimes\mathbb{C}_x)=
H^i(\mathbb{P}^1,r\mathcal{O}_{\mathbb{P}^1})=0,\ \ i\ge1,
$$
$$
H^i(Q_{n-2},\tilde{E}_{K_j}(-\sigma^*D)\otimes\mathbb{C}_x)=
H^i(\mathbb{P}^1,r\mathcal{O}_{\mathbb{P}^1}(-1))=0,\ \ i\ge0.
$$
This, together with the Base-change Theorem for $f_j$, shows that
$R^if_{j*}\tilde{E}_{K_j}=0,\ \ i\ge1,\ \ \
R^if_{j*}(\tilde{E}_{K_j}(-\sigma^*D))=0,\ \ i\ge0$.
Hence the Leray spectral sequence for the projection $f_j$ yields
\begin{equation}\label{H^iEKj}
H^i(K_j,\tilde{E}_{K_j})=0,\ \ i\ge1,\ \ \
H^i(K_j,\tilde{E}_{K_j}(-\sigma^*D))=0,\ \ i\ge0,\ \ \ j=1,2.
\end{equation}

Next, one uses the embedded in $\mathbb{P}^n$ blow-up of the cone $Q_{n-1}(t_j)$ that
$\sigma_*\mathcal{O}_{K_j}=\mathcal{O}_{Q_{n-1}(t_j)}$ and $R^i\sigma_*\mathcal{O}_{K_j}=0,\ i\ge1$.
Therefore, setting
$\tilde{E}_{t_j}:=\tilde{E}|_{Q_{n-1}(t_j)}$,
we have by the projection formula:
$\sigma_*\tilde{E}_{K_j}=\tilde{E}_{t_j}$,\
$R^i\sigma_*\tilde{E}_{K_j}=0,\ i\ge1$,
and
$\sigma_*(\tilde{E}_{K_j}(-\sigma^*D))=\tilde{E}_{t_j}(-D)$,\
$R^i\sigma_*(\tilde{E}_{K_j}(-\sigma^*D))=0,\ i\ge1$.
Now the Leray spectral sequence applied to $\sigma$
shows in view of (\ref{H^iEKj}) that
\begin{equation}\label{again h^i=}
\begin{array}{cc}
H^i(Q_{n-1}(t_j),\tilde{E}_{t_j})=H^i(K_j,\tilde{E}_{K_j})=0, \ & i\ge1,\\
H^i(Q_{n-1}(t_j),\tilde{E}_{t_j}(-D))=H^i(K_j,\tilde{E}_{K_j}(-\sigma^*D))=0, &
\ \ \ \ i\ge0,\ \ \ j=1,2.
\end{array}
\end{equation}
The equalities (\ref{h^i=}) and (\ref{again h^i=})
yield via base change for the flat morphism $\pi$
\begin{equation}\label{2 relations}
R^i\pi_*\tilde{E}=0,\ i\ge1,\ \ \
R^i\pi_*(\tilde{E}(-D))=0,\ \ i\ge0.
\end{equation}
The same argument yields base change isomorphisms
\begin{equation}\label{1base ch}
b_t:\ \pi_*\tilde{E}\otimes\mathbb{C}_t\overset{\simeq}\to
H^0(Q_{n-1}(t),\tilde{E}|_{Q_{n-1}(t)}),\ \ \ t\in\mathbb{P}^1.
\end{equation}

Consider the divisor $D=Q_{n-2}\times\mathbb{P}^1$ on $\tilde{Q}_n$ (see diagram
(\ref{Lefschetz pencil})) and the projections
$Q_{n-2}\overset{pr_1}\leftarrow Q_{n-2}\times\mathbb{P}^1\overset{pr_2}\to\mathbb{P}^1$.
By definition,
$\tilde{E}|_D=pr_1^*(\tilde{E}_m|_{Q_{n-2}})$, hence, since $\tilde{E}|_{Q_{n-2}}$
is trivial, the base change for the flat morphism $pr_2$ gives the isomorphisms
\begin{equation}\label{b'}
b':\ \pr_{2*}(\tilde{E}|_D)\overset{\simeq}\to H^0(Q_{n-2},E|_{Q_{n-2}})\otimes\mathcal{O}_{\mathbb{P}^1}
\simeq\mathbb{C}^r\otimes\mathcal{O}_{\mathbb{P}^1},
\end{equation}
\begin{equation}\label{2base ch}
b'_t:\ \pr_{2*}(\tilde{E}|_D)\otimes\mathbb{C}_t\overset{\simeq}\to
H^0(Q_{n-2},E|_{Q_{n-2}})\simeq\mathbb{C}^r,\ \ \ t\in\mathbb{P}^1.
\end{equation}
Now consider the exact triple
\begin{equation}\label{triple D}
0\to\tilde{E}(-D)\to\tilde{E}\to E|_D\to0
\end{equation}
and its restriction onto a fibre $Q_{n-1}(t)$ of the projection $\pi$ over an arbitrary
point $t\in\mathbb{P}^1$
\begin{equation}\label{triple D res t}
0\to\tilde{E}(-D)|_{Q_{n-1}(t)}\to\tilde{E}|_{Q_{n-1}(t)}\to
(E|_{Q_{n-2}})\otimes\mathbb{C}_t\to0.
\end{equation}
Applying the functor $R^i\pi_*$ to (\ref{triple D}) and using (\ref{2 relations}) and (\ref{b'}) we obtain the isomorphism of sheaves
\begin{equation}\label{rD}
r_D:\ \pi_*\tilde{E}\overset{\sim}\to\pr_{2*}(\tilde{E}|_D)\simeq
\mathbb{C}^r\otimes\mathcal{O}_{\mathbb{P}^1}.
\end{equation}
In particular, $\pi_*\tilde{E}$ is a trivial bundle.
Respectively, passing to cohomology of the exact sequence (\ref{triple D res t})
and using (\ref{h^i=}), (\ref{again h^i=}) and (\ref{2base ch}), we obtain the isomorphisms
\begin{equation}\label{res t}
res_t:H^0(Q_{n-1}(t),\tilde{E}|_{Q_{n-1}(t)})\overset{\sim}\to
H^0(Q_{n-2},E|_{Q_{n-2}}),\ \ \ t\in\mathbb{P}^1.
\end{equation}
By construction the isomorphisms (\ref{1base ch}), (\ref{2base ch}), (\ref{rD}) and
(\ref{res t}) fit in the commutative diagram
\begin{equation}\label{bch diagram}
\xymatrix{
\pi_*\tilde{E}\otimes\mathbb{C}_t\ar[d]_{b_t}^{\simeq}\ar[r]^-{r_D\otimes\mathbb{C}_t}_-{\simeq} &
\pr_{2*}(\tilde{E}|_D)\otimes\mathbb{C}_t\ar[d]_-{b'_t}^{\simeq}\\
H^0(Q_{n-1}(t),\tilde{E}|_{Q_{n-1}(t)})\ar[r]^-{res_t}_-{\simeq} &
H^0(Q_{n-2},E|_{Q_{n-2}})=\mathbb{C}^r}
\end{equation}
for $t\in\mathbb{P}^1$.
Next, since $E|_{Q_{n-2}}$ is trivial, the evaluation map
$H^0(Q_{n-2},E|_{Q_{n-2}})\otimes\mathcal{O}_{Q_{n-2}}\to E|_{Q_{n-2}}$
is an isomorphism, so that its composition
$e_t$ with the restriction
$H^0(Q_{n-2},E|_{Q_{n-2}})\otimes\mathcal{O}_{Q_{n-1}(t)}\twoheadrightarrow
H^0(Q_{n-2},E|_{Q_{n-2}})\otimes\mathcal{O}_{Q_{n-2}}$
is an epimorphism for any $t\in\mathbb{P}^1$ and fits in the commutative diagram
\begin{equation}\label{2bch diagram}
\xymatrix{
H^0(Q_{n-1}(t),\tilde{E}|_{Q_{n-1}(t)})\otimes\mathcal{O}_{Q_{n-1}(t)}\ar[d]_{ev_t}
\ar[r]^-{\pi^*res_t}_-{\simeq} &
H^0(Q_{n-2},E|_{Q_{n-2}})\otimes\mathcal{O}_{Q_{n-1}(t)}\ar@{>>}[d]_-{e_t}\\
\tilde{E}|_{Q_{n-1}(t)}\ar@{>>}[r]^-{res_{Q_{n-2}}} & E|_{Q_{n-2}}.}
\end{equation}
Here we understand $Q_{n-2}$ as lying in $Q_{n-1}(t)$ as a divisor. In particular, through any point of
$Q_{n-1}(t)\smallsetminus Q_{n-2}$ there passes a line, say, $l$ interesecting $Q_{n-2}$ at a point, say $y$.
Therefore, since $\tilde{E}|_l$ is trivial, we have a commutative diagram of restriction maps
$$
\xymatrix{
H^0(Q_{n-1}(t),\tilde{E}|_{Q_{n-1}(t)})\ar[d]_{\rho}
\ar[r]^-{res_t}_-{\simeq} &
H^0(Q_{n-2},E|_{Q_{n-2}})\ar[d]^-{\simeq}\\
H^0(l,\tilde{E}|_l)\ar[r]^-{\simeq} & H^0(y,\tilde{E}|_y).}
$$
Hence, $\rho$ is an isomorphism, and therefore the evaluation morphism $ev_t$ in (\ref{2bch diagram}) is an isomorphism
of sheaves. Composing it with the isomorphism
$\pi^*b_t:\pi^*\pi_*\tilde{E}|_{Q_{n-1}(t)}\overset{\simeq}\to
H^0(Q_{n-1}(t),\tilde{E}|_{Q_{n-1}(t)})\otimes\mathcal{O}_{Q_{n-1}(t)}$
arising from the left vertical isomorphism in (\ref{bch diagram}) we obtain the (evaluation) isomorphism
$ev|_{Q_{n-1}(t)}:\ \pi^*\pi_*\tilde{E}|_{Q_{n-1}(t)}\overset{\simeq}\to\tilde{E}|_{Q_{n-1}(t)}$.
Since this is true for any $t\in\mathbb{P}^1$, we obtain the isomorphism
$ev:\ \pi^*\pi_*\tilde{E}\overset{\simeq}\to\tilde{E}$ which together with (\ref{rD}) leads to the
triviality of $\tilde{E}$. Since clearly $\sigma_{Q*}\mathcal{O}_{\tilde{Q}}=\mathcal{O}_Q$, it follows
that $E=\sigma_{Q*}\tilde{E}=
\sigma_{Q*}(r\mathcal{O}_{\tilde{Q}})=r\mathcal{O}_Q$, i. e. we obtain
the statement of Proposition.
\end{proof}

\begin{proposition}\label{proposition7.4}
Let $E$ be a linearly trivial vector bundle on $GO(k,V)$ or $GS(k,V)$. Then $E$ is trivial.
\end{proposition}

\begin{proof}
Consider the case $GO(k,V)$. We give a proof by induction under the assumption that
$n:=\frac{\dim V}{2}\in \mathbb{Z}_{>0}.$ The case when $\dim V$ is odd can be treated similarly.

For $n=2$ we have
$GO(1,V)\simeq\mathbb{P}^1\times\mathbb{P}^1$, $GO(2,V)\simeq\mathbb{P}^1$,
and for these varieties our claim clearly holds. Therefore we assume that $n\ge3$ and argue
by induction on $k$.
If $k=1$, $GO(k,V)$ is a $(2n-2)$-dimensional quadric in $\mathbb{P}^{2n-1}$ so our statement holds by
Proposition \ref{trivial on lines}. Now let $1\le k\le n-2$,
and recall the graph of incidence $\Sigma$ with natural projections
\begin{equation}\label{equation108}
GO(k,V)\overset{q}\leftarrow\Sigma\overset{p}\to GO(k+1,V)
\end{equation}
(see subsection \ref{5.1}).

Let $E$ be a linearly trivial vector bundle on $GO(k+1,V)$.
Then the bundle $p^*E$ is linearly trivial on the fibres of $q$.
Since these fibres are quadrics, Proposition \ref{trivial on lines} implies that
$p^*E$ is trivial on the fibres of $q$. Furthermore, Proposition \ref{trivial on fibres}
yields an isomorphism $q^*q_*p^*E\overset{\simeq}\to p^* E$.
Hence, since $p^*E$ is trivial along the fibres of $p$ which are mapped by $q$
isomorphically to projective spaces $\mathbb{P}^k$ on $GO(k,V)$, it follows that $q_*p^*E$ is trivial along these projective subspaces $\mathbb{P}^k$ of $GO(k,V)$. Consequently,
$q_*p^*E$ is linearly trivial on $GO(k,V)$. Thus, by the induction assumption,
$q_*p^*E$ is trivial. Hence $p^*E$ and $E=p_* p^* E$ are trivial.

It remains to consider $GO(n,V)$. Here we employ induction on $n$.
For $n=3$ $GO(n,V)\simeq\mathbb{P}^3$, hence the statement holds in this case.
For $n\ge4$, consider the graph of incidence
$\Pi_n:=\{(V_1,V_{n})\in Q_{2n-2}\times GO(n,V)\ |\ V_1\subset V_{n}\}$
with natural projections
\begin{equation}\label{equation109}
Q_{2n-2}\overset{p}\leftarrow\Pi_n\overset{q}\to GO(n,V).
\end{equation}

Let $E$ be a linearly trivial vector bundle on $GO(n,V)$. Then $q^* E$ is trivial on
lines lying in the fibres of $p$ which are isomorphic to $GO(n-1,\mathbb{C}^{2n-2})$.
Hence $q^* E$ is trivial along the fibres of $p$ by the induction assumption.
Next, Proposition 7.1 yields an isomorphism $p^*p_*q^*E\overset{\simeq}\to q^*E$.
Since $q^* E$ is trivial on the fibres of $p$, it follows that $p_*q^*E$
is trivial on the projective subspaces $\mathbb{P}^{n-1}$ of the quadric $Q_{2n-2}$.
Therefore $p_* q^* E$ is trivial on the lines in $Q_{2n-2}$, so it is trivial by Proposition 7.3. Finally, $q^*E\simeq p^*p_*q^*E$ and $E= q_* q^* E$ are trivial as well.

Proceed to the case  of $GS(k,V)$. Substituting $GO$ by $GS$ in diagram (\ref{equation108}),
we obtain a diagram
$GS(k,V)\overset{q}\leftarrow\Sigma'\overset{p}\to GS(k+1,V)$,
where  $p$ is a $\mathbb{P}^k$-bundle and $q$ is a $\mathbb{P}^{2n-2k-1}$-bundle.
Respectively, substituting $GO$ by $GS$, and $Q_{2n-2}$ by $\mathbb{P}(V_n)$
in diagram (\ref{equation109}), we obtain a diagram
$\mathbb{P}(V_n) \overset{p}\leftarrow\Pi'_n\overset{q}\to GS_n$.
This enables us to carry out an argument very similar to the one for $GO(k,V)$.
\end{proof}

\end{document}